\documentclass[reqno]{amsart}

\usepackage{amssymb}
\usepackage{amsmath}
\usepackage{amsfonts}
\usepackage[colorlinks]{hyperref}
\usepackage{mathtools}
\usepackage{graphicx}
\usepackage{mathptmx}
\usepackage{bm}
\usepackage{physics}
\usepackage{csquotes}
\usepackage[thinc]{esdiff}
\usepackage{enumerate}

\numberwithin{equation}{section}

\theoremstyle{definition}
\newtheorem{definition}{Definition}
\newtheorem{lemma}{Lemma}
\newtheorem{theorem}{Theorem}

\newtheorem{example}{Example}
\theoremstyle{remark}
\newtheorem{remark}{Remark}

\allowdisplaybreaks

\newcommand{\assumptionRef}[1]{\textrm{Assumption (\ref{#1})}}
\newcommand{\assumRef}[1]{\textrm{Assum.\,}\br{\textrm{\ref{#1}}}}
\newcommand{\axiomRef}[1]{\textrm{Axiom (\ref{#1})}}
\newcommand{\ball}[3]{K_{#1}\sbr{#2,#3}}
\newcommand{\Bint}[3]{\br{B}\!\displaystyle\int_{#1}^{#2}{#3}}
\newcommand{\bF}{\mathbf{F}}
\newcommand{\bLambda}{\boldsymbol{\Lambda}}
\newcommand{\br}[1]{\left(#1\right)}
\newcommand{\bu}{\mathbf{u}}
\newcommand{\bv}{\mathbf{v}}
\newcommand{\bvarphi}{\boldsymbol{\varphi}}

\newcommand{\bx}{\mathbf{x}}
\newcommand{\C}[2]{C\br{#1,#2}}
\newcommand{\cbr}[1]{\left\{#1\right\}}
\newcommand{\cl}[1]{\clOperator\br{#1}}
\newcommand{\conv}[1]{\convOperator\br{#1}}
\newcommand{\D}[1]{D^1#1}
\newcommand{\DC}[3]{{}^{C}\!D^{#1}_{#2}{#3}}
\newcommand{\DCp}[4]{\frac{{}^{C}\!\partial^{#1}_{#2}{#3}}{\partial{#4}^{#1}}}
\newcommand{\defRef}[1]{\textrm{Def.\,\ref{#1}}}
\newcommand{\definitionRef}[1]{\textrm{Definition \ref{#1}}}
\newcommand{\defeq}{\vcentcolon=}
\newcommand{\DRL}[3]{{}^{RL}\!D^{#1}_{#2}{#3}}
\newcommand{\ds}{\mathrm{d}s}
\newcommand{\dt}{\mathrm{d}t}
\newcommand{\dtau}{\mathrm{d}\tau}
\newcommand{\eps}{\varepsilon}
\newcommand{\eqdef}{=\vcentcolon}
\newcommand{\eqText}[1]{\stackrel{#1}{=}}
\newcommand{\exampleRef}[1]{\textrm{Example \ref{#1}}}
\newcommand{\fF}{\mathcal{F}}
\newcommand{\fM}{\mathcal{M}}
\newcommand{\fN}{\mathcal{N}}
\newcommand{\JRL}[3]{J^{#1}_{#2}{#3}}
\newcommand{\LOne}[2]{L^1\br{#1,#2}}
\newcommand{\lemmaRef}[1]{\textrm{Lemma \ref{#1}}}
\newcommand{\leqText}[1]{\stackrel{#1}{\leq}}
\newcommand{\Lint}[3]{\br{L}\!\displaystyle\int_{#1}^{#2}{#3}}
\newcommand{\lmRef}[1]{\textrm{Lm.\,\ref{#1}}}
\newcommand{\lsrbr}[1]{\left[#1\right)}
\newcommand{\N}{\mathbb{N}}
\newcommand{\R}{\mathbb{R}}
\newcommand{\remarkRef}[1]{\textrm{Remark \,\ref{#1}}}
\newcommand{\remRef}[1]{\textrm{Rem.\,\ref{#1}}}
\newcommand{\sbr}[1]{\left[#1\right]}
\newcommand{\sdelta}{\delta^{\star}}
\newcommand{\sE}{E^{\star}}
\newcommand{\sectionRef}[1]{\textrm{Section \ref{#1}}}
\newcommand{\sgn}[1]{\sgnOperator\br{#1}}
\newcommand{\sI}{I^{\star}}
\newcommand{\st}{\!:}
\newcommand{\su}{u^{\star}}
\newcommand{\subseteqText}[1]{\stackrel{#1}{\subseteq}}
\newcommand{\theoremRef}[1]{\textrm{Theorem \ref{#1}}}
\newcommand{\thmRef}[1]{\textrm{Thm.\,\ref{#1}}}
\newcommand{\Z}{\mathbb{Z}}

\DeclareMathOperator{\clOperator}{cl}
\DeclareMathOperator{\convOperator}{conv}
\DeclareMathOperator{\sgnOperator}{sgn}

\begin{document}
	\title[On the existence of solutions of fractional differential equations in Banach spaces]{On the existence of solutions of fractional differential equations in Banach spaces}
	\author[D. Oberta]{Dušan Oberta}
	\address{Institute of Mathematics, Faculty of Mechanical Engineering, Brno University of Technology, Technická 2, 616 69 Brno, Czech Republic}
	\email{Dusan.Oberta@vutbr.cz, oberta.du@gmail.com}
	
	\begin{abstract}
		Utilising the notion of measures of non-compactness and Kamke function of order $\alpha$, we address the question of solvability of fractional differential equations in Banach spaces. In particular, we provide sufficient conditions ensuring the existence of a local solution. Our main existence theorem is then applied on countable systems of fractional differential equations arising from semi-discretisation of fractional PDEs with $p$-Laplacian.
	\end{abstract}
	
	\maketitle
	
	\section{Introduction}\label{sec_01}
	\setcounter{section}{1}
	\setcounter{equation}{0}
	Differential equations in Banach spaces have been extensively studied in last decades, e.g. in \cite{Deimling}, \cite{Ladas_Lakshmikantham}, \cite{Lakshmikantham_Leela}. Due to lack of compactness in infinite dimensional Banach spaces, classical results known for $\R^n$ are no longer valid. A special case of ODEs in Banach spaces are countable systems of classical ODEs, which can be viewed as ODEs in appropriate sequence spaces. These countable systems appear in many areas (such as stochastic processes or (semi-)discretisation of PDEs) and were studied e.g. in \cite{Banas_paper_ExistenceTheorem}, \cite{Banas_paper_SolvabilityOfInifiniteSystems}, \cite{Banas_book_sequences}, \cite{Banas_paper_ExistenceOfSolutions}, \cite{Han_Kloeden}, \cite{Samoilenko}, \cite{Slavik_multidimensional}, \cite{Slavik}.
	
	As in the case of $\R^n$, natural generalisation of ODEs in Banach spaces is in the form of fractional differential equations (FDEs) in Banach spaces. A simple example of application of the theory of FDEs in Banach spaces is a semi-discretisation of the anomalous diffusion equation (see e.g. \cite{Evangelista}, \cite{Meerschaert}, \cite{Metzler_Klafter}), which is of the form
	\begin{equation}\label{eq_01_01}
		\DCp{\alpha}{a}{u}{t}\br{t,x}=\diffp[2]{u}{x}\br{t,x},
	\end{equation}
	with the time fractional Caputo derivative of order $\alpha$ on the LHS. Moreover, such countable systems (i.e. countable systems of the same form as those arising from semi-discretisation of \eqref{eq_01_01}) are also studied directly, e.g. in \cite{Liu}.
	
	Basic results for classical FDEs can be found e.g. in \cite{Diethelm}, \cite{Kilbas}. However, one should note that the area of FDEs is still of a great interest nowadays. Even more relevant is the area of FDEs in Banach spaces, where some basic existence results have been proved only quite recently, and the theory is still largely incomplete. One of the goals of this paper is to fill in some of these gaps and to revise some of the existing results even for the integer order case. Some recent papers concerning the theory of FDEs in Banach spaces are \cite{Abbas}, \cite{Lakshmikantham_Devi}, \cite{Salem}, \cite{Zhou_Jiao_Pecaric}.
	
	Let $E$ be a real Banach space. In this paper we investigate the initial value problem
	\begin{subequations}\label{eq_01_02}
		\begin{align}
			\DC{\alpha}{a}{u\br{t}}&=f\br{t,u\br{t}},\quad{t}\in\sbr{a,b},\\
			u\br{a}&=u_0,
		\end{align}
	\end{subequations}
	where $\DC{\alpha}{a}{\br{\cdot}}$ denotes the Caputo fractional differential operator of order $0<\alpha<1$ and $f\st\sbr{a,b}\times{E}\rightarrow{E}$ is some appropriate function. We prove an existence theorem for \eqref{eq_01_02}, utilising the notion of a general measure of non-compactness and a general Kamke function of order $\alpha$ (newly introduced in \sectionRef{sec_02_07}). Our existence theorem is of a different type than the above-mentioned results (in the sense that it is not a subcase of any of the other results, nor contains any of these results as its subcase). 
	
	In \sectionRef{sec_02} we introduce a basic theory of fractional calculus in Banach spaces, measures of non-compactness, Kamke function of order $\alpha$ and prove some auxiliary results necessary for our main theorem. Moreover, we thoroughly establish the equivalence between solutions of the initial value problem \eqref{eq_01_02}, and solutions of the corresponding Volterra-type integral equation, since we were unable to find such a relation derived properly in the existing literature for the case of Banach spaces.
	
	In \sectionRef{sec_03} we prove our main existence theorem, which extends the known results for classical ODEs in Banach spaces to the case of FDEs in Banach spaces. These classical results can be found e.g. in \cite{Ambrosetti}, \cite{Banas_book}, \cite{Li}, \cite{Monch}. The proof of our main theorem follows the main ideas of the proof for the integer order case, mainly from \cite{Banas_book}. We support the proof with numerous details, which we were unable to find in the existing literature even for the integer order case. Some auxiliary results proved in \sectionRef{sec_02} are new in the fractional case and might be found useful in further development of the theory. Furthermore, in comparison to the integer order case from \cite{Banas_book}, we require an additional assumption, namely the measure of non-compactness $\mu$ having the so-called singleton property (notion we introduced in \sectionRef{sec_02_06}). This is further discussed in \sectionRef{sec_05}. At the end of \sectionRef{sec_03}, we prove also the basic existence and uniqueness theorem. It follows the very same arguments as the Picard theorem for classical ODEs in $\R^n$. However, since we feel the lack of explicit statement and proof of such a theorem in the existing literature, we provide the proof for completeness. Although, note that a similar result in the context of Hilfer fractional derivatives can be found in \cite{Mei_Peng_Gao}.
	
	In \sectionRef{sec_04} we provide a non-trivial example of a Kamke function of order $\alpha$, namely a superlinear function, which extends the case of a linear function usually found in the literature. Furthermore, note that this extension is not only in the context of superlinearity, but also in the context of Kamke function being dependent on the order $\alpha$. Finally, using our main theorem in the space $c_0$, we provide specific sufficient conditions for the existence of a local solution of \eqref{eq_01_02} in the space $c_0$. This result is then applied on a semi-discrete system obtained from a generalised version of \eqref{eq_01_01}, with the Laplacian on the RHS being replaced by a $p$-Laplacian and a non-linear perturbation being added to the RHS.
	
	Finally, in \sectionRef{sec_05} we discuss some possibilities regarding further development of the theory of FDEs in Banach spaces discussed in this paper.

	\section{Auxiliary Results}\label{sec_02}
	\setcounter{section}{2}
	\setcounter{equation}{0}
	Throughout this paper, let $E=\br{E,\norm{\cdot}_E}$ denote a real Banach space. The neutral element in $E$ will be denoted by $0_E$, or simply by $0$, if that is clear from the context. The closed ball centred at $x_0\in{E}$ with radius $r>0$ will be denoted by $\ball{E}{x_0}{r}$. Dual space of $E$, i.e. the space of all the bounded linear functionals on $E$, will be denoted by $\sE$.
	
	For $X\subseteq{E}$, let $\cl{X}$ denote the closure of $X$. Similarly, let $\conv{X}$ denote the convex hull of $X$. Recall that $\cl{\conv{X}}$, the closed convex hull of $X$, is the smallest closed and convex set containing $X$.
	
	Moreover, we define $\fM_E$ as the family of all the non-empty and bounded subsets of $E$. Similarly, we define $\fN_E$ as the family of all the non-empty and relatively compact subsets of $E$ (i.e. $\fN_E$ is the family of all the non-empty subsets of $E$ with compact closure).
	
	Finally, note that $\bx=\br{x_k}_{k=1}^{\infty}$ denotes a real sequence. Recall that the space $c_0\defeq\cbr{\bx\st\lim_{k\rightarrow\infty}x_k=0}$, equipped with the norm $\norm{\bx}_{c_0}\defeq\sup_{k=1,2,\dots}\abs{x_k}$ is a Banach space.
	
	\subsection{Abstract functions}
	An abstract function is a mapping $f\st\sbr{a,b}\rightarrow{E}$. For a family $X$ of abstract functions defined on $\sbr{a,b}$ and for $t\in\sbr{a,b}$, we define $X\br{t}\defeq\cbr{f\br{t}\st{f}\in{X}}\subseteq{E}$. Similarly, for some $Y\subseteq{Z}$ (where $Z$ is an arbitrary set) and a mapping $g\st{Z}\rightarrow{E}$, we define $g\br{Y}\defeq\cbr{g\br{x}\st{x}\in{Y}}\subseteq{E}$. We denote by $\C{\sbr{a,b}}{E}$ the space of all the continuous abstract functions $f\st\sbr{a,b}\rightarrow{E}$. Recall the well-known result that the space $\C{\sbr{a,b}}{E}$ equipped with the supremum norm $\norm{\cdot}_{\infty}$ is a Banach space.
	
	The derivative of an abstract function $f$ is meant in the strong sense, i.e. the derivative of $f$ at the point $t_0$, denoted by $\D{f}\br{t_0}$, is defined as the unique element of $E$ satisfying
	\begin{equation}\label{eq_02_01}
		\lim_{h\rightarrow0}\norm{\frac{f\br{t_0+h}-f\br{t_0}}{h}-\D{f}\br{t_0}}_E=0.
	\end{equation}
	
	\subsection{Integrals of abstract functions}
	For $f\in\C{\sbr{a,b}}{E}$, the abstract Riemann integral is defined in the same way as the classical Riemann integral. For more details, see e.g. Chapter $1.3$ in \cite{Ladas_Lakshmikantham}. Note that the abstract Riemann integral can be defined for a broader family of functions, but for our purpose, $\C{\sbr{a,b}}{E}$ is enough.
	
	Bochner integral is a generalisation of Lebesgue integral for abstract functions. Thorough construction of Bochner integral can be found in \cite{Mikusinski}. Note that for $E=\R$, Bochner and Lebesgue integral coincide.
	
	We denote by $\LOne{\sbr{a,b}}{E}$ the space of all the Bochner integrable functions $f\st\sbr{a,b}\rightarrow{E}$. Note that the space $\LOne{\sbr{a,b}}{\R}$ is the classical space of Lebesgue integrable functions. For $f\in\LOne{\sbr{a,b}}{E}$ and $g\in\LOne{\sbr{a,b}}{\R}$, we denote the Bochner and Lebesgue integral of $f$ and $g$, respectively, as
	\begin{equation*}
		\Bint{a}{b}{f\br{t}\dt}\quad\text{and}\quad\Lint{a}{b}{g\br{t}\dt}.
	\end{equation*}
	
	Moreover, one should keep in mind that for $f\in\C{\sbr{a,b}}{E}$, the abstract Riemann and Bochner integral coincide. Next, we provide some important properties of Bochner integral, crucial for our purpose.
	\begin{theorem}[\normalfont{\cite{Mikusinski}, Theorem $3.1$ in Chapter III}]\label{thm_01}
		Let $f\in\LOne{\sbr{a,b}}{E}$. Then it holds that $\norm{f\br{\cdot}}_E\in\LOne{\sbr{a,b}}{\R}$. Moreover
		\begin{equation*}
			\norm{\Bint{a}{b}{f\br{t}\dt}}_E\leq\Lint{a}{b}{\norm{f\br{t}}_E\,\dt}.
		\end{equation*}
	\end{theorem}
	\begin{theorem}[\normalfont{\cite{Mikusinski}, Theorem $2.2$ in Chapter XIII}]\label{thm_02}
		Let $f\in\C{\sbr{a,b}}{E}$. Then for all $t\in\sbr{a,b}$ it holds that
		\begin{equation*}
			\D{\br{\Bint{a}{\br\cdot}{f\br{s}\ds}}}\br{t}=f\br{t}.
		\end{equation*}
	\end{theorem}
	\begin{theorem}[\normalfont{\cite{Mikusinski}, Theorem $2.3$ in Chapter XIII}]\label{thm_03}
		Let $f\st\sbr{a,b}\rightarrow{E}$ be continuously differentiable on $\sbr{a,b}$. Then it holds that
		\begin{equation*}
			\Bint{a}{b}{\D{f}\br{t}\dt}=f\br{b}-f\br{a}.
		\end{equation*}
	\end{theorem}
	
	\subsection{Fractional calculus for abstract functions}\label{sec_02_03}
	In this section, we briefly introduce the theory of fractional calculus in Banach spaces. The theory is analogous to the case of real-valued functions, which is thoroughly described in \cite{Diethelm}. We remind some basic results, which can be found e.g. in \cite{Anastassiou}, as well as prove some new ones, which are necessary for our purpose, but we were unable to find those in the existing literature.
	
	Recall that $E$ denotes a real Banach space. Furthermore, note that $\Gamma\br{\cdot}$ denotes the \textit{Gamma function}. Recall its well known property, namely that $\Gamma\br{x+1}=x\cdot\Gamma\br{x}$, for all $x>0$. 
	\begin{definition}[\normalfont{Riemann-Liouville fractional abstract integral operator of order $\alpha$}]\label{def_01}
		Let $\alpha>0$ and $f\in\LOne{\sbr{a,b}}{E}$. The \textit{Riemann-Liouville fractional abstract integral operator of order $\alpha$}, denoted by $\JRL{\alpha}{a}{\br{\cdot}}$, is defined as
		\begin{equation*}
			\JRL{\alpha}{a}{f}\br{t}\defeq\frac{1}{\Gamma\br{\alpha}}\Bint{a}{t}{\br{t-s}^{\alpha-1}f\br{s}\ds},
		\end{equation*}
		for $a\leq{t}\leq{b}$. Furthermore, we define $\JRL{0}{a}{f}\defeq{f}$ (i.e. the identity operator).
	\end{definition}
	\begin{remark}
		Note that for $\alpha=1$, $\JRL{\alpha}{a}{\br{\cdot}}$ coincides with the classical Bochner integral operator.
	\end{remark}
	\begin{remark}\label{rem_01}
		Note that in \definitionRef{def_01}, the value $\JRL{\alpha}{a}{f}\br{a}$ is meant as the limit, i.e. $\JRL{\alpha}{a}{f}\br{a}\defeq\lim_{t\rightarrow{a^+}}\JRL{\alpha}{a}{f}\br{t}$.
	\end{remark}
	\begin{theorem}[\normalfont{\cite{Anastassiou}, Theorem $2.3$}]
		Let $\alpha\geq0$ and $f\in\LOne{\sbr{a,b}}{E}$. Then $\JRL{\alpha}{a}{f}\br{t}$ exists for almost all $t\in\sbr{a,b}$. Moreover, $\JRL{\alpha}{a}{f}\in\LOne{\sbr{a,b}}{E}$.
	\end{theorem}
	\begin{theorem}[\normalfont{\cite{Anastassiou}, Remark $2.4$ and Theorem $2.5$}]\label{thm_05}
		Let $\alpha\geq0$ and $f\in\C{\sbr{a,b}}{E}$. Then $\JRL{\alpha}{a}{f}\in\C{\sbr{a,b}}{E}$.
	\end{theorem}
	\begin{theorem}[\normalfont{\cite{Anastassiou}, Theorem $2.6$}]\label{thm_06}
		Let $\alpha_1,\alpha_2\geq0$ and $f\in\LOne{\sbr{a,b}}{E}$. Then for almost all $t\in\sbr{a,b}$ it holds that
		\begin{equation}\label{eq_02_02}
			\JRL{\alpha_1}{a}{\JRL{\alpha_2}{a}{f}}\br{t}=\JRL{\alpha_2}{a}{\JRL{\alpha_1}{a}{f}}\br{t}=\JRL{\alpha_1+\alpha_2}{a}{f}\br{t}.
		\end{equation}
		Moreover, if $f\in\C{\sbr{a,b}}{E}$ or $\alpha_1+\alpha_2\geq1$, then the identity \eqref{eq_02_02} holds for all $t\in\sbr{a,b}$.
	\end{theorem}
	\begin{remark}\label{rem_03}
		Further on in this paper, we always assume $0<\alpha<1$. Although, one should note that all the arguments, as well as the statement of our main theorem, remain valid also for the case of $\alpha=1$.
	\end{remark}
	\begin{definition}[\normalfont{Riemann-Liouville fractional abstract differential operator of order $\alpha$}]\label{def_02}
		The \textit{Riemann-Liouville fractional abstract differential operator of order $\alpha$}, denoted by $\DRL{\alpha}{a}{\br{\cdot}}$, is defined as
		\begin{equation}\label{eq_02_03}
			\DRL{\alpha}{a}{f}\defeq\D{\JRL{1-\alpha}{a}{f}},
		\end{equation}
		whenever the RHS of \eqref{eq_02_03} is well-defined. Furthermore, we define $\DRL{0}{a}{f}\defeq{f}$ (i.e. the identity operator).
	\end{definition}
	\begin{remark}
		Note that $\DRL{\alpha}{a}{\br{\cdot}}$ can be analogously defined also for $\alpha\geq1$. Moreover, note that for $\alpha=1$, such an operator would coincide with the classical differential operator as defined by \eqref{eq_02_01} (hence the notation $\D{\br{\cdot}}$). Even though the case $\alpha\geq1$ is not needed for our purpose, one should keep in mind \remarkRef{rem_03}.
	\end{remark}
	\begin{theorem}\label{thm_07}
		Let $f\in\C{\sbr{a,b}}{E}$. Then
		\begin{equation*}
			\DRL{1-\alpha}{a}{\DRL{\alpha}{a}{\JRL{1}{a}{f}}}=f.
		\end{equation*}
	\end{theorem}
	\begin{proof}
	It follows that
	\begin{equation*}
		\begin{aligned}
			\DRL{1-\alpha}{a}{\DRL{\alpha}{a}{\JRL{1}{a}{f}}}&\eqText{\defRef{def_02}}\D{\JRL{\alpha}{a}{\D{\JRL{1-\alpha}{a}{\JRL{1}{a}{f}}}}}\eqText{\thmRef{thm_06}}\D{\JRL{\alpha}{a}{\D{\JRL{1}{a}{\JRL{1-\alpha}{a}{f}}}}}\\
			&\eqText{\thmRef{thm_02}}\D{\JRL{\alpha}{a}{\JRL{1-\alpha}{a}{f}}}\eqText{\thmRef{thm_06}}\D{\JRL{1}{a}{f}}\eqText{\thmRef{thm_02}}f.
		\end{aligned}
	\end{equation*}
	Note that the former use of \theoremRef{thm_02} is justified due to \theoremRef{thm_05}.
	\end{proof}

	\begin{theorem}\label{thm_08}
		Let $f\in\C{\sbr{a,b}}{E}$. Then
		\begin{equation*}
			\DRL{\alpha}{a}{\JRL{\alpha}{a}{f}}=f.
		\end{equation*}
	\end{theorem}
	\begin{proof}
	It holds that
	\begin{equation*}
		\DRL{\alpha}{a}{\JRL{\alpha}{a}{f}}\eqText{\defRef{def_02}}\D{\JRL{1-\alpha}{a}{\JRL{\alpha}{a}{f}}}\eqText{\thmRef{thm_06}}\D{\JRL{1}{a}{f}}.
	\end{equation*}
	The claim then follows immediately from \theoremRef{thm_02}.
	\end{proof}

	\begin{definition}[\normalfont{Caputo fractional abstract differential operator of order $\alpha$}]\label{def_03}
		The \textit{Caputo fractional abstract differential operator of order $\alpha$}, denoted by $\DC{\alpha}{a}{\br{\cdot}}$, is defined as
		\begin{equation}\label{eq_02_04}
			\DC{\alpha}{a}{f}\defeq\DRL{\alpha}{a}{\br{f\br{\cdot}-f\br{a}}},
		\end{equation}
		whenever the RHS of \eqref{eq_02_04} is well-defined. Furthermore, we define $\DC{0}{a}{f}\defeq{f}$ (i.e. the identity operator).
	\end{definition}
	\begin{remark}
		It can be shown that under certain assumptions on $f$ it holds that
		\begin{equation}\label{eq_02_05}
			\DC{\alpha}{a}{f}=\JRL{1-\alpha}{a}{\D{f}}.
		\end{equation}
		However, \definitionRef{def_03} is more suitable for applications than \eqref{eq_02_05}, as it covers a broader family of functions.
	\end{remark}
	
	\subsection{Initial value problem for fractional differential equations}\label{sec_02_04}
	Let $u_0\in{E}$ and $f\st\sbr{a,b}\times{E}\rightarrow{E}$. Consider the initial value problem (IVP)
	\begin{subequations}\label{eq_02_06}
		\begin{align}
			\DC{\alpha}{a}{u}\br{t}&=f\br{t,u\br{t}},\quad{t}\in\sbr{a,b},\label{eq_02_06_01}\\
			u\br{a}&=u_0.\label{eq_02_06_02}
		\end{align}
	\end{subequations}
	By a solution of \eqref{eq_02_06}, we mean $u\in\C{\sbr{a,b}}{E}$ such that $\DC{\alpha}{a}{u}\in\C{\sbr{a,b}}{E}$, \eqref{eq_02_06_01} holds for all $t\in\sbr{a,b}$ and the initial condition \eqref{eq_02_06_02} is satisfied. In the next theorem we show that solving the IVP \eqref{eq_02_06} is under certain assumptions equivalent to solving the corresponding Volterra-type integral equation.
	\begin{theorem}\label{thm_09}
		Consider the IVP \eqref{eq_02_06} and let $f$ be continuous. Then $u\in\C{\sbr{a,b}}{E}$ is a solution of \eqref{eq_02_06} if and only if it satisfies for all $t\in\sbr{a,b}$ the following Volterra-type integral equation
		\begin{equation}\label{eq_02_07}
			u\br{t}=u_0+\JRL{\alpha}{a}{f\br{\cdot,u\br{\cdot}}}\br{t}.
		\end{equation}
	\end{theorem}
	\begin{proof}
	\enquote{$\implies$} Let $u\in\C{\sbr{a,b}}{E}$ be a solution of \eqref{eq_02_06}. Define $z\st\sbr{a,b}\rightarrow{E}$ as
	\begin{equation}\label{eq_02_08}
		z\br{t}\defeq{f}\br{t,u\br{t}},\quad{t}\in\sbr{a,b}.
	\end{equation}
	Obviously, $z\in\C{\sbr{a,b}}{E}$, since both $f$ and $u$ are continuous. Let $t\in\sbr{a,b}$ be arbitrary. Substituting \eqref{eq_02_08} into \eqref{eq_02_06_01} and using \eqref{eq_02_06_02}, we obtain
	\begin{equation}\label{eq_02_09}
		z\br{t}=\DC{\alpha}{a}{u}\br{t}\eqText{\defRef{def_03}}\DRL{\alpha}{a}{\br{u\br{\cdot}-u_0}}\br{t}\eqText{\defRef{def_02}}\D{\JRL{1-\alpha}{a}{\br{u\br{\cdot}-u_0}}}\br{t}.
	\end{equation}
	
	Now, we apply the $\JRL{1}{a}{\br{\cdot}}$ operator on both sides of \eqref{eq_02_09}. Since $z\br{\cdot}$ is continuous, we may use \theoremRef{thm_03} to obtain
	\begin{equation}\label{eq_02_10}
		\begin{aligned}
			\JRL{1}{a}{z}\br{t}&=\JRL{1}{a}{\D{\JRL{1-\alpha}{a}{\br{u\br{\cdot}-u_0}}}}\br{t}\eqText{\thmRef{thm_03}}\JRL{1-\alpha}{a}{\br{u\br{\cdot}-u_0}}\br{t}-\JRL{1-\alpha}{a}{\br{u\br{\cdot}-u_0}}\br{a}\\
			&=\JRL{1-\alpha}{a}{\br{u\br{\cdot}-u_0}}\br{t},
		\end{aligned}
	\end{equation}
	where the last equality follows from
	\begin{equation}\label{eq_02_11}
		\begin{aligned}
			\Gamma\br{1-\alpha}&\cdot\norm{\JRL{1-\alpha}{a}{\br{u\br{\cdot}-u_0}}\br{a}}_E\eqText{\remRef{rem_01}}\norm{\lim_{t\rightarrow{a}^+}\Bint{a}{t}{\br{t-s}^{-\alpha}\cdot\br{u\br{s}-u_0}\ds}}_E\\
			&\leqText{\thmRef{thm_01}}\max_{a\leq{s}\leq{b}}\norm{u\br{s}-u_0}_E\cdot\lim_{t\rightarrow{a}^+}\Lint{a}{t}{\br{t-s}^{-\alpha}\ds}=0.
		\end{aligned}
	\end{equation}
	
	Finally, since $u$ is continuous, applying \theoremRef{thm_07} on $\br{u\br{\cdot}-u_0}\br{t}$ yields
	\begin{equation*}
		\begin{aligned}
			u\br{t}-u_0&\eqText{\thmRef{thm_07}}\DRL{1-\alpha}{a}{\DRL{\alpha}{a}{\JRL{1}{a}{\br{u\br{\cdot}-u_0}}}}\br{t}\\
			&\eqText{\thmRef{thm_06}}\DRL{1-\alpha}{a}{\DRL{\alpha}{a}{\JRL{\alpha}{a}{\JRL{1-\alpha}{a}{\br{u\br{\cdot}-u_0}}}}}\br{t}\eqText{\eqref{eq_02_10}}\DRL{1-\alpha}{a}{\DRL{\alpha}{a}{\JRL{\alpha}{a}{\JRL{1}{a}{z}}}}\br{t}\\
			&\eqText{\thmRef{thm_06}}\DRL{1-\alpha}{a}{\DRL{\alpha}{a}{\JRL{1}{a}{\JRL{\alpha}{a}{z}}}}\br{t}\eqText{\thmRef{thm_05}\,\&\,\thmRef{thm_07}}\JRL{\alpha}{a}{z}\br{t}.
		\end{aligned}
	\end{equation*}
	Keeping in mind the definition \eqref{eq_02_08}, \eqref{eq_02_07} follows immediately.
	
	\enquote{$\impliedby$} Let $u\in\C{\sbr{a,b}}{E}$ satisfy \eqref{eq_02_07} for all $t\in\sbr{a,b}$. Let $t\in\sbr{a,b}$ be arbitrary. Since $f\br{\cdot,u\br{\cdot}}$ is continuous, applying \theoremRef{thm_08} yields
	\begin{equation*}
		f\br{t,u\br{t}}\eqText{\thmRef{thm_08}}\DRL{\alpha}{a}{\JRL{\alpha}{a}{f\br{\cdot,u\br{\cdot}}}}\br{t}\eqText{\eqref{eq_02_07}}\DRL{\alpha}{a}{\br{u\br{\cdot}-u\br{a}}}\br{t}\eqText{\defRef{def_03}}\DC{\alpha}{a}{u}\br{t}.
	\end{equation*}
	Thus, $u\br{\cdot}$ satisfies \eqref{eq_02_06_01} for all $t\in\sbr{a,b}$. In the view of continuity of $f\br{\cdot,u\br{\cdot}}$, this also means that $\DC{\alpha}{a}{u}\in\C{\sbr{a,b}}{E}$. Furthermore, by the same arguments as in \eqref{eq_02_11}, one can easily verify that $u\br{\cdot}$ satisfies \eqref{eq_02_06_02} as well.
	\end{proof}
	
	\subsection{Basic results from functional analysis}
	Now, we remind some standard results from functional analysis, as well as prove a new one, which is needed for our purpose. Note that the third theorem is one of the many consequences of the Hahn-Banach theorem.
	\begin{theorem}[\normalfont{Arzelà-Ascoli theorem}]\label{thm_10}
		Let $X\subseteq\C{\sbr{a,b}}{\R}$. Then $X$ is relatively compact in $\C{\sbr{a,b}}{\R}$ if and only if $X$ is uniformly bounded and uniformly equicontinuous on $\sbr{a,b}$.
	\end{theorem}
	\begin{theorem}[\normalfont{Schauder fixed point theorem}]\label{thm_11}
		Let $X\subseteq{E}$ be non-empty, closed, convex and bounded. Let $\fF\st{X}\rightarrow{X}$ be a continuous mapping such that $\fF\br{X}$ is relatively compact in $E$. Then there exists a fixed point of $\fF$ in $X$.
	\end{theorem}
	\begin{theorem}[\normalfont{Hahn-Banach theorem, \cite{Kesavan}, Theorem $2.5$}]\label{thm_12}
		Let $X,Y\subseteq{E}$ be non-empty and disjoint convex sets. Assume that $X$ is closed and $Y$ is compact. Then $X$ and $Y$ can be strictly separated by a hyperplane, i.e.
		\begin{equation*}
			\exists{f}\in\sE\,\exists{\alpha}\in\R\,\exists\eps>0\st{f}\br{x}\leq\alpha-\eps\textnormal{ and }f\br{y}\geq\alpha+\eps,\quad\forall{x}\in{X},\forall{y}\in{Y}.
		\end{equation*}
	\end{theorem}
	\begin{remark}
		The next lemma can be viewed as a generalisation of the mean value theorem (which, as it is well-known, does not hold in the same form in a general Banach space as it does in $\R$) for Banach spaces. Without the singular kernel, it is a known result, and it can be found in differential form e.g. as Proposition $2.1$ in \cite{Deimling}. However, as we were not able to find such a result with singular kernel in the existing literature, we provide a detailed proof ourselves.
	\end{remark}
	\begin{lemma}\label{lm_01}
		Let $f\in\C{\sbr{a,b}}{E}$. Then for all $t\in\sbr{a,b}$ it holds that
		\begin{equation}\label{eq_02_12}
			\frac{\Bint{a}{t}{\br{t-s}^{\alpha-1}f\br{s}\ds}}{\Lint{a}{t}{\br{t-s}^{\alpha-1}\ds}}\in\cl{\conv{\cbr{f\br{s}\st{a}\leq{s}\leq{t}}}}.
		\end{equation}
	\end{lemma}
	\begin{proof}
	It is obviously enough to show the statement for $t=b$. Define $K$ as
	\begin{equation}\label{eq_02_13}
		K\defeq\cl{\conv{\cbr{f\br{t}\st{a}\leq{t}\leq{b}}}}.
	\end{equation}
	Firstly, we show that
	\begin{equation}\label{eq_02_14}
		K=\bigcap_{\lambda\in\R,h\in\sE\,\st\cbr{x\in{E}\,\st{h}\br{x}\leq\lambda}\supseteq{K}}\cbr{x\in{E}\st{h}\br{x}\leq\lambda}\eqdef{S}.
	\end{equation}
	
	\enquote{$\subseteq$} Since all the sets on the RHS of \eqref{eq_02_14} contain $K$, also their intersection does.
	
	\enquote{$\supseteq$} It is enough to show that if $x\notin{K}$, then $x\notin{S}$. Let $x\in{E}$ such that $x\notin{K}$ be arbitrary. Applying \theoremRef{thm_12} (on the closed, convex set $K$ and compact, convex set $\cbr{x}$), we obtain
	\begin{equation*}
		\exists{h}_0\in\sE\,\exists\lambda_0\in\R\st{h}_0\br{y}\leq\lambda_0\text{ and }h_0\br{x}>\lambda_0,\quad\forall{y}\in{K}.
	\end{equation*}
	This means that $x\notin{S}$, thus we are done.
	
	Next, we show that for $g\in\C{\sbr{a,b}}{\R}$ it holds that (compare with \eqref{eq_02_12})
	\begin{equation}\label{eq_02_15}
		\frac{\Bint{a}{b}{g\br{t}f\br{t}\dt}}{\Lint{a}{b}{g\br{t}\dt}}\in\cl{\conv{\cbr{f\br{t}\st{a}\leq{t}\leq{b}}}}\eqText{\eqref{eq_02_13}}K.
	\end{equation}
	Note that since $f$ and $g$ are both continuous, both the integrals in \eqref{eq_02_15} reduce to Riemann integrals (to be precise, to abstract Riemann integral and Riemann integral, respectively). Define $A$ as the denominator of the LHS of \eqref{eq_02_15}, i.e.
	\begin{equation}\label{eq_02_16}
		A\defeq\Lint{a}{b}{g\br{t}\dt}.
	\end{equation}
	
	Consider a partition $a=\tau_0<\tau_1<\dots<\tau_n=b$ of $\sbr{a,b}$ and define $\nu$ as
	\begin{equation*}
		\nu\defeq\max_{j=1,\dots,n}\br{\tau_j-\tau_{j-1}}.
	\end{equation*}
	Consider an arbitrary sequence of such partitions for which $\nu\rightarrow0^+$. By the definition of abstract Riemann integral, we readily obtain
	\begin{equation}\label{eq_02_17}
		\Bint{a}{b}{g\br{t}f\br{t}\dt}=\lim_{\nu\rightarrow0^+}\sum_{j=1}^n\br{\tau_j-\tau_{j-1}}\cdot{g}\br{\tau_j}f\br{\tau_j}.
	\end{equation}
	Let $\lambda_0\in\R$ and $h_0\in\sE$ be arbitrary such that $K\subseteq\cbr{x\in{E}\st{h}_0\br{x}\leq\lambda_0}$. Then
	\begin{equation}\label{eq_02_18}
		\begin{aligned}
			&h_0\br{\frac{1}{A}\cdot\Bint{a}{b}{g\br{t}f\br{t}\dt}}\eqText{\eqref{eq_02_17}}h_0\br{\frac{1}{A}\cdot\lim_{\nu\rightarrow0^+}\sum_{j=1}^n\br{\tau_j-\tau_{j-1}}\cdot{g}\br{\tau_j}f\br{\tau_j}}\\
			&\qquad\qquad\eqText{h_0\text{ is linear and continuous}}\frac{1}{A}\cdot\lim_{\nu\rightarrow0^+}\sum_{j=1}^n\br{\tau_j-\tau_{j-1}}\cdot{g}\br{\tau_j}\cdot{h}_0\br{f\br{\tau_j}}\\
			&\qquad\qquad\leqText{f\br{\tau_j}\in{K}\subseteq\cbr{x\in{E}\,\st{h}_0\br{x}\leq\lambda_0}}\lambda_0\cdot\frac{1}{A}\cdot\lim_{\nu\rightarrow0^+}\sum_{j=1}^n\br{\tau_j-\tau_{j-1}}\cdot{g}\br{\tau_j}\eqText{\eqref{eq_02_16}}\lambda_0.
		\end{aligned}
	\end{equation}
	Since $\lambda_0\in\R$ and $h_0\in\sE$ were chosen as arbitrary (satisfying the additional condition), \eqref{eq_02_15} follows immediately from \eqref{eq_02_14} and \eqref{eq_02_18}.
	
	Let $\eps>0$ be arbitrary. From \eqref{eq_02_15} we readily obtain
	\begin{equation}\label{eq_02_19}
		J\br{\eps}\defeq\frac{\Bint{a}{b-\eps}{\br{b-t}^{\alpha-1}f\br{t}\dt}}{\Lint{a}{b-\eps}{\br{b-t}^{\alpha-1}\dt}}\in\cl{\conv{\cbr{f\br{t}\st{a}\leq{t}\leq{b}-\eps}}}\subseteq{K}.
	\end{equation}
	Notice that in the limit $\eps\rightarrow0^+$, $J\br{\eps}$ converges to the LHS of \eqref{eq_02_12}. Moreover, since in \eqref{eq_02_19}, the relation $J\br{\eps}\in{K}$ holds for all $\eps>0$, and since $K$ is closed, we immediately obtain
	\begin{equation*}
		\lim_{\eps\rightarrow0^+}J\br{\eps}\in{K}\eqText{\eqref{eq_02_13}}\cl{\conv{\cbr{f\br{t}\st{a}\leq{t}\leq{b}}}}.
	\end{equation*}
	Hence proved.
	\end{proof}
	
	\subsection{Measures of non-compactness}\label{sec_02_06}
	In this section, we briefly introduce the concept of measures of non-compactness. For our purpose, the axiomatic approach is well-suited and is thoroughly described in Chapter $3$ in \cite{Banas_book}. Firstly, recall that $\fM_E$ denotes the family of all the non-empty and bounded subsets of $E$, whereas $\fN_E$ denotes the family of all the non-empty and relatively compact subsets of $E$.
	\begin{remark}
		Note that for $X,Y\subseteq{E}$ and $\lambda_1,\lambda_2\in\R$, the linear combination $\lambda_1{X}+\lambda_2{Y}$ is defined as
		\begin{equation*}
			\lambda_1{X}+\lambda_2{Y}\defeq\cbr{\lambda_1{x}+\lambda_2{y}\st{x}\in{X},y\in{Y}}.
		\end{equation*}
	\end{remark}
	\begin{definition}[\normalfont{Measure of non-compactness}]\label{def_04}
		The mapping $\mu\st\fM_E\rightarrow\lsrbr{0,\infty}$ is called a \textit{measure of non-compactness in $E$}, if it satisfies the following conditions
		\begin{enumerate}[(i)]
			\item The family $\ker\mu\defeq\cbr{X\in\fM_E\st\mu\br{X}=0}$ is non-empty and $\ker\mu\subseteq\fN_E$.\label{def_04_ax_01}
			\item $\mu\br{\cl{X}}=\mu\br{X}$, for all $X\in\fM_E$.
			\item If $X\subseteq{Y}$, then $\mu\br{X}\leq\mu\br{Y}$, for all $X,Y\in\fM_E$.\label{def_04_ax_03}
			\item $\mu\br{\conv{X}}=\mu\br{X}$, for all $X\in\fM_E$.
			\item $\mu\br{\lambda{X}+\br{1-\lambda}Y}\leq\lambda\cdot\mu\br{X}+\br{1-\lambda}\cdot\mu\br{Y}$,\, for all $X,Y\in\fM_E$,\, for all $\lambda\in\sbr{0,1}$.
			\item Let $\Lambda$ be an index set. If $\cbr{X_{\lambda}}_{\lambda\in\Lambda}\subseteq\fM_E$ is a non-empty family of non-empty, closed and bounded sets totally ordered by inclusion and if for any $\eps>0$ there exists $\lambda\in\Lambda$ such that $\mu\br{X_{\lambda}}\leq\eps$, then $X_{\infty}\defeq\bigcap_{\lambda\in\Lambda}X_{\lambda}\neq\emptyset$ and $X_{\infty}\in\ker\mu$.\label{def_04_ax_06}
		\end{enumerate}
	\end{definition}
	\begin{definition}\label{def_05}
		We say that a measure of non-compactness $\mu$ is \textit{sublinear}, if
		\begin{enumerate}[(i)]\addtocounter{enumi}{6}
			\item $\mu\br{\lambda{X}}=\abs{\lambda}\cdot\mu\br{X}$, for all $\lambda\in\R$, for all $X\in\fM_E$.
			\item $\mu\br{X+Y}\leq\mu\br{X}+\mu\br{Y}$, for all $X,Y\in\fM_E$.
		\end{enumerate}
	\end{definition}
	\begin{definition}\label{def_06}
		We say that a measure of non-compactness $\mu$ satisfies the \textit{singleton property}, if
		\begin{enumerate}[(i)]\addtocounter{enumi}{8}
			\item $\mu\br{\cbr{x}}=0$, for all $x\in{E}$.
		\end{enumerate}
	\end{definition}
	\begin{remark}
		Notion of the singleton property for a measure of non-compactness was introduced by us, as we find it to be a crucial assumption in the main existence theorem.
	\end{remark}
	\begin{remark}\label{rem_09}
		Note that the singleton property does not follow from the axioms of \definitionRef{def_04}, nor from the axioms of \definitionRef{def_05}. Indeed, it is quite easy to verify that the mapping $\mu\st{X}\longmapsto\sup_{x\in{X}}\norm{x}_E$, for any $X\in\fM_E$, is a sublinear measure of non-compactness, but does not have the singleton property, since $\mu\br{X}=0$ if and only if $X=\cbr{0_E}$ (i.e. $\ker\mu=\cbr{\cbr{0_E}}$). This example, originally as an example of a measure of non-compactness with the kernel consisting of a single set, was taken from Section $3.1$ in \cite{Banas_book}.
	\end{remark}
	\begin{definition}
		The \textit{Hausdorff measure of non-compactness}, denoted by $\chi\br{\cdot}$, is defined as
		\begin{equation}\label{eq_02_20}
			\chi\br{X}\defeq\inf\cbr{\eps>0\st{X}\text{ has a finite $\eps$-net in }E},
		\end{equation}
		for every $X\in\fM_E$.
	\end{definition}
	\begin{remark}\label{rem_10}
		According to \cite{Banas_book_sequences}, $\chi\br{\cdot}$ is a sublinear measure of non-compactness satisfying the singleton property (see Lemma $5.9$ and Lemma $5.10$ in \cite{Banas_book_sequences}). Moreover, one of the interesting properties of the Hausdorff measure of non-compactness is that $\chi\br{X}=0$ if and only if $X\in\fN_E$.
	\end{remark}
	\begin{remark}
		Even though we have a general formula for the Hausdorff measure of non-compactness, the form \eqref{eq_02_20} is not suitable for applications. Nevertheless, for some particular Banach spaces, the Hausdorff measure of non-compactness is well-described. That is the case of e.g. sequence spaces, for which such measures are thoroughly described in \cite{Banas_book_sequences}.
	\end{remark}
	\begin{theorem}[\normalfont{\cite{Banas_book_sequences}, Theorem $5.18$}]\label{thm_13}
		The Hausdorff measure of non-compactness in the space $c_0$ is of the form
		\begin{equation}\label{eq_02_21}
			\chi\br{X}=\lim_{k\rightarrow\infty}\sup_{\bx\in{X}}\sup_{l\geq{k}}\abs{x_l},
		\end{equation}
		for every $X\in\fM_E$.
	\end{theorem}
	\begin{remark}
		Recall that $X\in\fM_{\C{\sbr{a,b}}{E}}$ denotes a non-empty and bounded set (with respect to the $\norm{\cdot}_{\infty}$ norm) of continuous functions $f\st\sbr{a,b}\rightarrow{E}$.
	\end{remark}
	\begin{lemma}[\normalfont{\cite{Banas_book}, Lemma $13.2.1$}]\label{lm_02}
		Let $\mu$ be a sublinear measure of non-compactness and $X\in\fM_{\C{\sbr{a,b}}{E}}$. Then for all $t,s\in\sbr{a,b}$ it holds that
		\begin{equation*}
			\abs{\mu\br{X\br{t}}-\mu\br{X\br{s}}}\leq\mu\br{\ball{E}{0_E}{1}}\cdot\sup_{x\in{X}}\,\sup_{\tau_1,\tau_2\in\sbr{a,b}\,\st\abs{\tau_1-\tau_2}\leq\abs{t-s}}\norm{x\br{\tau_1}-x\br{\tau_2}}_E.
		\end{equation*}
	\end{lemma}
	\begin{remark}
		Note that for $f\in\LOne{\sbr{a,b}}{\R}$ and $X\in\fM_{\C{\sbr{a,b}}{E}}$, we define
		\begin{equation*}
			\Bint{a}{b}{f\br{t}\cdot{X}\br{t}\dt}\defeq\cbr{\Bint{a}{b}{f\br{t}\cdot{u}\br{t}\dt}\st{u}\in{X}}.
		\end{equation*}
	\end{remark}
	\begin{remark}
		The next lemma is an extension of Lemma $13.2.2$ in \cite{Banas_book} for integrals with singular kernel. Since we were not able to find such a result in the existing literature, we provide a detailed proof ourselves.
	\end{remark}
	\begin{lemma}\label{lm_03}
		Let $\mu$ be a sublinear measure of non-compactness. Let $X\in\fM_{\C{\sbr{a,b}}{E}}$ be uniformly equicontinuous. Then for all $t\in\sbr{a,b}$ it holds that
		\begin{equation}\label{eq_02_22}
			\mu\br{\Bint{a}{t}{\br{t-s}^{\alpha-1}\cdot{X}\br{s}\ds}}\leq\Lint{a}{t}{\br{t-s}^{\alpha-1}\cdot\mu\br{X\br{s}}\ds}.
		\end{equation}
	\end{lemma}
	\begin{proof}
	It is obviously enough to show the statement for $t=b$. Due to the fact that the set $X\subseteq\C{\sbr{a,b}}{E}$ is uniformly equicontinuous, it holds that
	\begin{equation}\label{eq_02_23}
		\begin{aligned}
			\forall\eps>0\,&\exists{a}=\tau_0<\tau_1<\dots<\tau_n=b\st\norm{u\br{t}-u\br{s}}_E\leq\eps;\\
			&\forall{u}\in{X};\,\forall{t},s\in\sbr{\tau_{j-1},\tau_j};\,\forall{j}=1,2,\dots,n.
		\end{aligned}
	\end{equation}
	Let $u\in{X}$ and $\eps>0$ be arbitrary. Consider the partition $a=\tau_0<\tau_1<\dots<\tau_n=b$ from \eqref{eq_02_23}. Then
	\begin{equation}\label{eq_02_24}
		\begin{aligned}
			&\norm{\Bint{a}{b}{\br{b-t}^{\alpha-1}u\br{t}\dt}-\sum_{j=1}^n\br{\Lint{\tau_{j-1}}{\tau_j}{\br{b-t}^{\alpha-1}\dt}\cdot{u}\br{\tau_j}}}_E\\
			&\qquad=\norm{\sum_{j=1}^n\br{\Bint{\tau_{j-1}}{\tau_j}{\br{b-t}^{\alpha-1}u\br{t}\dt}}-\sum_{j=1}^n\br{\Bint{\tau_{j-1}}{\tau_j}{\br{b-t}^{\alpha-1}u\br{\tau_j}\dt}}}_E\\
			&\qquad\leqText{\thmRef{thm_01}}\sum_{j=1}^n\br{\Lint{\tau_{j-1}}{\tau_j}{\br{b-t}^{\alpha-1}\cdot\norm{u\br{t}-u\br{\tau_j}}_E\,\dt}}\\
			&\qquad\leqText{\eqref{eq_02_23}}\eps\cdot\Lint{a}{b}{\br{b-t}^{\alpha-1}\dt}=\frac{1}{\alpha}\br{b-a}^{\alpha}\cdot\eps\eqdef{A}\cdot\eps.
		\end{aligned}
	\end{equation}
	Since $u\in{X}$ was chosen as arbitrary, we readily obtain (by adding a \enquote{clever zero})
	\begin{equation*}
		\begin{aligned}
			&\Bint{a}{b}{\br{b-t}^{\alpha-1}\cdot{X}\br{t}\dt}\\
			&\qquad\subseteq\cbr{\Bint{a}{b}{\br{b-t}^{\alpha-1}u\br{t}\dt}-\sum_{j=1}^n\br{\Lint{\tau_{j-1}}{\tau_j}{\br{b-t}^{\alpha-1}\dt}\cdot{u}\br{\tau_j}}\st{u}\in{X}}+\\
			&\qquad\qquad\qquad+\cbr{\sum_{j=1}^n\br{\Lint{\tau_{j-1}}{\tau_j}{\br{b-t}^{\alpha-1}\dt}\cdot{u}\br{\tau_j}}\st{u}\in{X}}\\
			&\qquad\subseteqText{\eqref{eq_02_24}}A\cdot\eps\cdot\ball{E}{0_E}{1}+\cbr{\sum_{j=1}^n\br{\Lint{\tau_{j-1}}{\tau_j}{\br{b-t}^{\alpha-1}\dt}\cdot{u}\br{\tau_j}}\st{u}\in{X}}.
		\end{aligned}
	\end{equation*}
	Due to \axiomRef{def_04_ax_03} of \definitionRef{def_04} and the sublinearity of $\mu$ (see \definitionRef{def_05}), it follows that
	\begin{equation}\label{eq_02_25}
		\begin{aligned}
			&\mu\br{\Bint{a}{b}{\br{b-t}^{\alpha-1}\cdot{X}\br{t}\dt}}\\
			&\qquad\qquad\leq{A}\cdot\eps\cdot\mu\br{\ball{E}{0_E}{1}}+\sum_{j=1}^n\br{\Lint{\tau_{j-1}}{\tau_j}{\br{b-t}^{\alpha-1}\dt}\cdot\mu\br{X\br{\tau_j}}}.
		\end{aligned}
	\end{equation}
	In the limit $\eps\rightarrow0^+$, the LHS of \eqref{eq_02_25} does not depend on $\eps$ and the first term on the RHS vanishes. Thus, we obtain (recall that $n$ depends on $\eps$)
	\begin{equation}\label{eq_02_26}
		\mu\br{\Bint{a}{b}{\br{b-t}^{\alpha-1}\cdot{X}\br{t}\dt}}\leq\lim_{\eps\rightarrow0^+}\sum_{j=1}^n\br{\Lint{\tau_{j-1}}{\tau_j}{\br{b-t}^{\alpha-1}\dt}\cdot\mu\br{X\br{\tau_j}}}.
	\end{equation}
	
	Moreover, from \lemmaRef{lm_02}, due to the uniform equicontinuity of $X$, it immediately follows that $\mu\br{X\br{\cdot}}$ is uniformly continuous on $\sbr{a,b}$. Using similar arguments as in \eqref{eq_02_24}, it is easy to verify that the RHS of \eqref{eq_02_26} converges to the RHS of \eqref{eq_02_22} (for $t=b$). Hence proved.
	\end{proof}

	\begin{lemma}[\normalfont{\cite{Banas_book}, Theorem $11.2$}]\label{lm_04}
		Let $\mu$ be a measure of non-compactness in $E$. Then the mapping $M\st\fM_{\C{\sbr{a,b}}{E}}\rightarrow\lsrbr{0,\infty}$ defined as
		\begin{equation*}
			M\br{X}\defeq\lim_{\eps\rightarrow0^+}\sup_{x\in{X}}\,\sup_{t,s\in\sbr{a,b}\,\st\abs{t-s}\leq\eps}\norm{x\br{t}-x\br{s}}_E+\sup_{t\in\sbr{a,b}}\br{\mu\br{X\br{t}}},
		\end{equation*}
		is a measure of non-compactness in the Banach space $\C{\sbr{a,b}}{E}$.
	\end{lemma}
	
	\subsection{Kamke function of order $\alpha$}\label{sec_02_07}
	Now, we introduce the notion of a Kamke function of order $\alpha$, which is inspired by \cite{Li} and the Kamke comparison condition in Chapter 13 in \cite{Banas_book}. Notice that in the fractional case, Kamke function is dependent on the order $\alpha$.
	\begin{definition}[\normalfont{Kamke function of order $\alpha$}]\label{def_08}
		We say that $w_{\alpha}\st\sbr{a,b}\times\lsrbr{0,\infty}\rightarrow\lsrbr{0,\infty}$ is a \textit{Kamke function of order $\alpha$ on $\sbr{a,b}$}, if
		\begin{enumerate}[(i)]
			\item $w_{\alpha}$ is continuous on $\sbr{a,b}\times\lsrbr{0,\infty}$.
			\item $w_{\alpha}\br{t,0}=0$, for all $t\in\sbr{a,b}$.\label{def_08_ax_02}
			\item $u\equiv0$ is the unique non-negative continuous solution of
			\begin{equation}\label{eq_02_27}
				u\br{t}\leq\frac{1}{\Gamma\br{\alpha}}\cdot\Lint{a}{t}{\br{t-s}^{\alpha-1}\cdot{w}_{\alpha}\br{s,u\br{s}}\ds},\quad{t}\in\sbr{a,b},
			\end{equation}
			for which
			\begin{equation}\label{eq_02_28}
				\lim_{t\rightarrow{a}^+}\frac{u\br{t}}{\br{t-a}^{\alpha}}=0.
			\end{equation}\label{def_08_ax_03}
		\end{enumerate}
	\end{definition}
	\begin{remark}
		Note that due to \axiomRef{def_08_ax_02} of \definitionRef{def_08}, $u\equiv0$ is always a solution of \eqref{eq_02_27}. The question is whether it is the unique solution satisfying the additional condition \eqref{eq_02_28}.
	\end{remark}

	\section{Main Results}\label{sec_03}
	\setcounter{section}{3}
	\setcounter{equation}{0}
	Recall that $E$ denotes a real Banach space and $0<\alpha<1$. Let $\sdelta>0$, $a\in\R$, $u_0\in{E}$ and define $\sI\defeq\sbr{a,a+\sdelta}$, $R\defeq\sI\times\ball{E}{u_0}{\beta}$, for some $\beta>0$. Let $f\st{R}\rightarrow{E}$. Further on, we will study the following IVP
	\begin{subequations}\label{eq_03_01}
		\begin{align}
			\DC{\alpha}{a}{u\br{t}}&=f\br{t,u\br{t}},\quad{t}\in\sI,\\
			u\br{a}&=u_0.\label{eq_03_01_02}
		\end{align}
	\end{subequations}
	
	\begin{theorem}\label{thm_14}
		Consider the IVP \eqref{eq_03_01}. Let $w_{\alpha}$ be a Kamke function of order $\alpha$ on $\sI$ and let $\mu$ be a sublinear measure of non-compactness in $E$ satisfying the singleton property (see \definitionRef{def_06}). Assume that
		\begin{enumerate}[(i)]
			\item $f$ is uniformly continuous on $R$.
			\item $\norm{f\br{t,u}}_E\leq{M}$, for all $\br{t,u}\in{R}$, for some $M>0$.\label{th_03_01_ass_02}
			\item $\mu\br{f\br{t,X}}\leq{w}_{\alpha}\br{t,\mu\br{X}}$, for all $t\in\sI$, for all $\emptyset\neq{X}\subseteq\ball{E}{u_0}{\beta}$.\label{th_03_01_ass_03}
		\end{enumerate}
		Then there exists a solution of \eqref{eq_03_01} on $I\defeq\sbr{a,a+\delta}$, where $\delta$ is defined as
		\begin{equation}\label{eq_03_02}
			\delta\defeq\min\cbr{\sdelta,\br{\frac{\beta}{M}\cdot\Gamma\br{\alpha+1}}^{\frac{1}{\alpha}}}.
		\end{equation}
	\end{theorem}
	\begin{proof}
	Before providing the proof itself, we firstly briefly describe its basic ideas, which should help clarifying some of the steps while going through the proof. The proof consists of the following main steps
	\begin{enumerate}[1.]
		\item A nested sequence $\cbr{X_k}_{k=1}^\infty$ of non-empty, bounded, closed, convex and uniformly equicontinuous subsets of $C\br{I,E}$ is constructed, with $X_{k+1}\defeq\cl{\conv{\fF\br{X_k}}}$, $X_0$ being chosen appropriately and $\fF$ being an integral operator defined as the RHS of the corresponding Volterra-type integral equation \eqref{eq_02_07}.
		\item The sequence $\cbr{v_k\br{\cdot}}_{k=1}^{\infty}\defeq\cbr{\mu\br{X_k\br{\cdot}}}_{k=1}^{\infty}$ (of real-valued functions of a single real variable) is shown to be uniformly equicontinuous and uniformly bounded on $I$. Thus, using the Arzelà-Ascoli theorem, it is shown that it converges uniformly to some function $v_{\infty}\br{\cdot}\in{C}\br{I,\R}$.
		\item Next, it is verified that $v_{\infty}\br{\cdot}$ satisfies the integral inequality together with the convergence requirement from \axiomRef{def_08_ax_03} of \definitionRef{def_08}, thus $v_{\infty}\equiv0$ on $I$.
		\item Finally, it is shown that the set $X_{\infty}\defeq\bigcap_{k=1}^{\infty}X_k$ together with the mapping $\fF$ satisfy the assumptions of the Schauder fixed point theorem, and thus the corresponding Volterra-type equation \eqref{eq_02_07} has a solution.
	\end{enumerate}
	
	Now, we provide details of these particular steps. Firstly, note that $\eqref{eq_03_02}$ implies
	\begin{equation}\label{eq_03_03}
		\frac{1}{\Gamma\br{\alpha+1}}M\delta^{\alpha}\leq\beta.
	\end{equation}
	
	Define $X_0$ as
	\begin{equation}\label{eq_03_04}
		\begin{aligned}
			X_0\defeq&\biggl\{u\st{u}\in\C{I}{E};\;\norm{u\br{\cdot}-u_0}_{\infty}\leq\beta;\;u\br{a}=u_0;\\
			&\phantom{\biggl\{u}\left.\norm{u\br{t}-u\br{s}}_E\leq\frac{2M}{\Gamma\br{\alpha+1}}\abs{t-s}^{\alpha},\forall{t,s}\in{I}\right\}.
		\end{aligned}
	\end{equation}
	Note that the last condition in \eqref{eq_03_04} states that $u\in{X}_0$ is Hölder continuous, with the Hölder constant being independent of $u$. It is easy to verify that $X_0$ is non-empty, bounded, closed and convex. Moreover, the last condition in \eqref{eq_03_04} immediately implies that the set $X_0$ is uniformly equicontinuous on $I$.
	
	Define the operator $\fF\st{X}_0\rightarrow\C{I}{E}$ as the RHS of the Volterra-type integral equation \eqref{eq_02_07}, i.e. as
	\begin{equation}\label{eq_03_05}
		\br{\fF\br{u}}\br{t}\defeq{u}_0+\frac{1}{\Gamma\br{\alpha}}\Bint{a}{t}{\br{t-s}^{\alpha-1}f\br{s,u\br{s}}\ds},\quad{t}\in{I}.
	\end{equation}
	Note that the definition \eqref{eq_03_05} is correct, since by \theoremRef{thm_05}, $\fF$ maps $\C{I}{E}$ into $\C{I}{E}$. Furthermore, by \theoremRef{thm_09}, it is enough to show that \eqref{eq_03_05} has a fixed point in $\C{I}{E}$.
	
	Firstly, we verify that $\fF\br{X_0}\subseteq{X}_0$. As for the boundedness condition in \eqref{eq_03_04}, let $u\in{X}_0$ and $t\in{I}$ be arbitrary. Then
	\begin{equation}\label{eq_03_06}
		\begin{aligned}
			&\norm{\br{\fF\br{u}}\br{t}-u_0}_E\leqText{\thmRef{thm_01}}\frac{1}{\Gamma\br{\alpha}}\Lint{a}{t}{\br{t-s}^{\alpha-1}\norm{f\br{s,u\br{s}}}_E\,\ds}\\
			&\qquad\leqText{\assumRef{th_03_01_ass_02}}\frac{M}{\Gamma\br{\alpha}}\Lint{a}{t}{\br{t-s}^{\alpha-1}\ds}=\frac{M}{\alpha\cdot\Gamma\br{\alpha}}\br{t-a}^{\alpha}\leq\frac{M\cdot\delta^{\alpha}}{\Gamma\br{\alpha+1}}\leqText{\eqref{eq_03_03}}\beta.
		\end{aligned}
	\end{equation}
	Regarding the initial value condition in \eqref{eq_03_04}, let $u\in{X}_0$ be arbitrary. In the view of \remarkRef{rem_01} and by similar arguments as in \eqref{eq_03_06}, we obtain
	\begin{equation}\label{eq_03_07}
		\begin{aligned}
			\norm{\br{\fF\br{u}}\br{a}-u_0}_E&\leqText{\thmRef{thm_01}}\frac{1}{\Gamma\br{\alpha}}\lim_{\eps\rightarrow0^+}\Lint{a}{a+\eps}{\br{a+\eps-s}^{\alpha-1}\norm{f\br{s,u\br{s}}}_E\,\ds}\\
			&\leqText{\assumRef{th_03_01_ass_02}}\frac{M}{\Gamma\br{\alpha+1}}\lim_{\eps\rightarrow0^+}\eps^{\alpha}=0.
		\end{aligned}
	\end{equation}
	Finally, as for the Hölder continuity condition in \eqref{eq_03_04}, let $u\in{X}_0$ and $t,s\in{I}$ such that ${t}\leq{s}$ be arbitrary. Then
	\begin{equation*}
		\begin{aligned}
			\Gamma\br{\alpha}&\cdot\norm{\br{\fF\br{u}}\br{t}-\br{\fF\br{u}}\br{s}}_E\\
			&=\norm{\Bint{a}{t}{\br{t-\tau}^{\alpha-1}f\br{\tau,u\br{\tau}}\dtau}-\Bint{a}{s}{\br{s-\tau}^{\alpha-1}f\br{\tau,u\br{\tau}}\dtau}}_E\\
			&\leq\norm{\Bint{a}{t}{\br{\br{t-\tau}^{\alpha-1}-\br{s-\tau}^{\alpha-1}}f\br{\tau,u\br{\tau}}\dtau}}_E+\\
			&\qquad\qquad+\norm{\Bint{t}{s}{\br{s-\tau}^{\alpha-1}f\br{\tau,u\br{\tau}}\dtau}}_E\\
			&\leqText{\thmRef{thm_01}\,\&\,\assumRef{th_03_01_ass_02}}M\left(\Lint{a}{t}{\br{\br{t-\tau}^{\alpha-1}-\br{s-\tau}^{\alpha-1}}\dtau}\right.+\\
			&\qquad\qquad+\Lint{t}{s}{\br{s-\tau}^{\alpha-1}\dtau}\biggr)\\
			&=\frac{M}{\alpha}\br{\br{t-a}^{\alpha}-\br{s-a}^{\alpha}+2\br{s-t}^{\alpha}}\\
			&\leqText{t\leq{s}}\frac{2M}{\alpha}\br{s-t}^{\alpha}=\Gamma\br{\alpha}\cdot\frac{2M}{\Gamma\br{\alpha+1}}\abs{s-t}^{\alpha}.
		\end{aligned}
	\end{equation*}
	Thus, $\fF\br{X_0}\subseteq{X}_0$. Now, define $\cbr{X_k}_{k=1}^{\infty}$ as
	\begin{equation}\label{eq_03_08}
		X_{k+1}\defeq\cl{\conv{\fF\br{X_k}}},\quad{k}=0,1,\dots
	\end{equation}
	Since $\fF\br{X_0}\subseteq{X_0}$ and $X_0$ is closed and convex, it is easy to inductively verify that
	\begin{equation}\label{eq_03_09}
		X_{k+1}\subseteq{X}_k,\quad{k}=0,1,\dots
	\end{equation}
	
	Define $v_k\st{I}\rightarrow\lsrbr{0,\infty}$ as
	\begin{equation}\label{eq_03_10}
		v_k\br{t}\defeq\mu\br{X_k\br{t}},\quad{t}\in{I},\quad{k}=0,1,\dots
	\end{equation}
	Directly from the \axiomRef{def_04_ax_03} of \definitionRef{def_04} and in the view of \eqref{eq_03_09}, for $k=0,1,\dots$ and for all $t\in{I}$ it follows that
	\begin{equation}\label{eq_03_11}
		0\leq{v}_{k+1}\br{t}\leq{v}_k\br{t}.
	\end{equation}
	Now, we verify that the family $\cbr{v_k\br{\cdot}}_{k=0}^\infty$ is uniformly equicontinuous on $I$. Let $k=0,1,\dots$ and $t,s\in{I}$ be arbitrary. Then
	\begin{equation*}
		\begin{aligned}
			\abs{v_k\br{t}-v_k\br{s}}&\leqText{\lmRef{lm_02}}\mu\br{\ball{E}{0_E}{1}}\cdot\sup_{x\in{X_k}}\,\sup_{\tau_1,\tau_2\in{I}\,\st\abs{\tau_1-\tau_2}\leq\abs{t-s}}\norm{x\br{\tau_1}-x\br{\tau_2}}_E\\
			&\leqText{X_k\subseteq{X}_0}\mu\br{\ball{E}{0_E}{1}}\cdot\sup_{x\in{X_0}}\,\sup_{\tau_1,\tau_2\in{I}\,\st\abs{\tau_1-\tau_2}\leq\abs{t-s}}\norm{x\br{\tau_1}-x\br{\tau_2}}_E.
		\end{aligned}
	\end{equation*}
	The uniform equicontinuity of $\cbr{v_k}_{k=0}^\infty$ (in the space $\C{I}{\R}$) follows directly from the uniform equicontinuity of $X_0$ (in the space $\C{I}{E}$). Moreover, since $v_0$ is continuous on the compact interval $I$, it is bounded on $I$, and thus in the view of \eqref{eq_03_11}, the family $\cbr{v_k}_{k=0}^\infty$ is uniformly bounded on $I$.
	
	Therefore, from \theoremRef{thm_10}, we get that $\cbr{v_k\br{\cdot}}_{k=0}^\infty$ is relatively compact in $\C{I}{\R}$. Thus, in the view of \eqref{eq_03_11}, it converges uniformly to some $v_{\infty}\br{\cdot}\geq0$, i.e.
	\begin{equation}\label{eq_03_12}
		\exists{v}_{\infty}\in\C{I}{\R}\st{v}_k\rightrightarrows{v}_{\infty}\geq0.
	\end{equation}
	
	Let us verify that $v_{\infty}$ satisfies the inequality \eqref{eq_02_27}. Firstly, note that it can be easily verified that for a uniformly equicontinuous family $H\subseteq\C{I}{E}$, the family $f\br{\cdot,H\br{\cdot}}=\cbr{f\br{\cdot,u\br{\cdot}}\st{u}\in{H}}\subseteq\C{I}{E}$ is also uniformly equicontinuous, due to the uniform continuity of $f$. Let $k=0,1,\dots$ and $t\in{I}$ be arbitrary. Then
	\begin{equation}\label{eq_03_13}
		\begin{aligned}
			v_{k+1}\br{t}&\eqText{\eqref{eq_03_08}\,\&\,\eqref{eq_03_10}}\mu\br{\cl{\conv{\fF\br{X_k\br{t}}}}}\eqText{\defRef{def_04}}\mu\br{\fF\br{X_k\br{t}}}\\
			&\eqText{\eqref{eq_03_05}}\mu\br{u_0+\frac{1}{\Gamma\br{\alpha}}\Bint{a}{t}{\br{t-s}^{\alpha-1}f\br{s,X_k\br{s}}\ds}}\\
			&\leqText{\defRef{def_05}\,\&\,\defRef{def_06}}\frac{1}{\Gamma\br{\alpha}}\cdot\mu\br{\Bint{a}{t}{\br{t-s}^{\alpha-1}f\br{s,X_k\br{s}}\ds}}\\
			&\leqText{\lmRef{lm_03}}\frac{1}{\Gamma\br{\alpha}}\cdot\Lint{a}{t}{\br{t-s}^{\alpha-1}\cdot\mu\br{f\br{s,X_k\br{s}}}\ds}\\
			&\leqText{\assumRef{th_03_01_ass_03}\,\&\,\eqref{eq_03_10}}\frac{1}{\Gamma\br{\alpha}}\cdot\Lint{a}{t}{\br{t-s}^{\alpha-1}{w}_{\alpha}\br{s,v_k\br{s}}\ds}.
		\end{aligned}
	\end{equation}
	Taking the (pointwise) limit in \eqref{eq_03_13}, we obtain
	\begin{equation}\label{eq_03_14}
		\begin{aligned}
			v_{\infty}\br{t}&\leq\lim_{k\rightarrow\infty}\frac{1}{\Gamma\br{\alpha}}\cdot\Lint{a}{t}{\br{t-s}^{\alpha-1}{w}_{\alpha}\br{s,v_k\br{s}}\ds}\\
			&=\frac{1}{\Gamma\br{\alpha}}\cdot\Lint{a}{t}{\br{t-s}^{\alpha-1}\cdot\lim_{k\rightarrow\infty}{w}_{\alpha}\br{s,v_k\br{s}}\ds}\\
			&\eqText{w\br{s,\cdot}\text{ is continuous}}\frac{1}{\Gamma\br{\alpha}}\cdot\Lint{a}{t}{\br{t-s}^{\alpha-1}{w}_{\alpha}\br{s,v_{\infty}\br{s}}\ds}.
		\end{aligned}
	\end{equation}
	Note that the exchange of the limit and integral in \eqref{eq_03_14} can be justified as follows
	\begin{equation}\label{eq_03_15}
		\begin{aligned}
			&\abs{\lim_{k\rightarrow\infty}\Lint{a}{t}{\br{t-s}^{\alpha-1}w_{\alpha}\br{s,v_k\br{s}}\ds}-\Lint{a}{t}{\br{t-s}^{\alpha-1}\cdot\lim_{k\rightarrow\infty}w_{\alpha}\br{s,v_k\br{s}}\ds}}\\
			&\qquad\qquad\leq\lim_{k\rightarrow\infty}\Lint{a}{t}{\br{t-s}^{\alpha-1}\cdot\abs{w_{\alpha}\br{s,v_k\br{s}}-w_{\alpha}\br{s,v_{\infty}\br{s}}}\,\ds}\\
			&\qquad\qquad\leq\frac{\br{t-a}^{\alpha}}{\alpha}\cdot\lim_{k\rightarrow\infty}\max_{a\leq{s}\leq{t}}\abs{w_{\alpha}\br{s,v_k\br{s}}-w_{\alpha}\br{s,v_{\infty}\br{s}}}=0,
		\end{aligned}
	\end{equation}
	where the last equality in \eqref{eq_03_15} follows immediately from $w_{\alpha}\br{\cdot,v_k\br{\cdot}}\rightrightarrows{w}_{\alpha}\br{\cdot,v_{\infty}\br{\cdot}}$ on $I$, which can be easily verified with the use of the uniform convergence $v_k\rightrightarrows{v_{\infty}}$ on $I$ and the continuity of $w_{\alpha}$.
	
	Now, we verify that $v_{\infty}$ satisfies also the condition \eqref{eq_02_28}. Note that due to \eqref{eq_03_11} and \eqref{eq_03_12}, it is enough to show that
	\begin{equation}\label{eq_03_16}
		\lim_{t\rightarrow{a}^+}\frac{v_1\br{t}}{\br{t-a}^{\alpha}}=0.
	\end{equation}
	Define $y\br{\cdot}$ as
	\begin{equation*}
		y\br{t}\defeq{u}_0+\frac{1}{\Gamma\br{\alpha+1}}\br{t-a}^{\alpha}\cdot{f}\br{a,u_0},\quad{t}\in{I}.
	\end{equation*}
	Let $x_1\in\fF\br{X_0}$ be arbitrary, i.e.
	\begin{equation*}
		x_1\br{t}=u_0+\frac{1}{\Gamma\br{\alpha}}\Bint{a}{t}{\br{t-s}^{\alpha-1}\cdot{f}\br{s,u\br{s}}\ds},\quad{t}\in{I},
	\end{equation*}
	for some $u\in{X}_0$. Note that it can be easily verified that for any $X\subseteq{E}$, it follows that $\sup_{x\in\cl{\conv{X}}}\norm{x}_E=\sup_{x\in{X}}\norm{x}_E$. Now, let $t\in{I}$ be arbitrary. Then
	\begin{equation}\label{eq_03_17}
		\begin{aligned}
			&\norm{y\br{t}-x_1\br{t}}_E\\
			&\qquad=\norm{\frac{1}{\Gamma\br{\alpha+1}}\br{t-a}^{\alpha}\cdot{f}\br{a,u_0}-\frac{1}{\Gamma\br{\alpha}}\Bint{a}{t}{\br{t-s}^{\alpha-1}\cdot{f}\br{s,u\br{s}}\ds}}_E\\
			&\qquad\leqText{\lmRef{lm_01}}\!\sup_{z\in\cl{\conv{\cbr{f\br{s,u\br{s}}\,\st{s}\in\sbr{a,t}}}}}\bigg\|\frac{1}{\Gamma\br{\alpha+1}}\br{t-a}^{\alpha}f\br{a,u_0}-\\
			&\qquad\qquad\qquad-\frac{1}{\Gamma\br{\alpha}}\Lint{a}{t}{\br{t-s}^{\alpha-1}\ds}\cdot{z}\bigg\|_E\\
			&\qquad=\frac{1}{\Gamma\br{\alpha+1}}\br{t-a}^{\alpha}\cdot\sup_{s\in\sbr{a,t}}\norm{f\br{a,u_0}-f\br{s,u\br{s}}}_E.
		\end{aligned}
	\end{equation}
	Define $\gamma\br{\cdot}$ as
	\begin{equation*}
		\gamma\br{t}\defeq\sup_{u\in{X}_0}\,\sup_{s\in\sbr{a,t}}\norm{f\br{a,u_0}-f\br{s,u\br{s}}}_E,\quad{t}\in{I}.
	\end{equation*}
	From \eqref{eq_03_17}, it is obvious that for all ${x_1\in\fF\br{X_0}}$ and for all $t\in{I}$ it follows that
	\begin{equation}\label{eq_03_18}
		\norm{y\br{t}-x_1\br{t}}_E\leq\frac{1}{\Gamma\br{\alpha+1}}\br{t-a}^{\alpha}\cdot\gamma\br{t}.
	\end{equation}
	Due to the Hölder continuity condition in \eqref{eq_03_04} and the uniform continuity of $f$, it can be easily shown that
	\begin{equation}\label{eq_03_19}
		\lim_{t\rightarrow{a}^+}\gamma\br{t}=0.
	\end{equation}
	
	Notice that if we take $x_1\in\conv{\fF\br{X_0}}$ instead of $x_1\in\fF\br{X_0}$, the conclusion \eqref{eq_03_18} still holds. Similarly, due to the continuity of $\norm{\cdot}_E$, the conclusion \eqref{eq_03_18} holds also for $x_1\in\cl{\conv{\fF\br{X_0}}}=X_1$. Thus, for all $t\in{I}$ it holds that
	\begin{equation}\label{eq_03_20}
		\begin{aligned}
			X_1\br{t}&\subseteq\ball{E}{y\br{t}}{\frac{1}{\Gamma\br{\alpha+1}}\br{t-a}^{\alpha}\cdot\gamma\br{t}}\\
			&=\cbr{y\br{t}}+\frac{1}{\Gamma\br{\alpha+1}}\br{t-a}^\alpha\cdot\gamma\br{t}\cdot\ball{E}{0_E}{1}.
		\end{aligned}
	\end{equation}
	Moreover, \eqref{eq_03_20} implies that
	\begin{equation}\label{eq_03_21}
		v_1\br{t}\eqText{\eqref{eq_03_10}}\mu\br{X_1\br{t}}\leqText{\defRef{def_05}\,\&\,\defRef{def_06}}\frac{1}{\Gamma\br{\alpha+1}}\br{t-a}^{\alpha}\cdot\gamma\br{t}\cdot\mu\br{\ball{E}{0_E}{1}}.
	\end{equation}
	From \eqref{eq_03_19} and \eqref{eq_03_21}, we immediately obtain \eqref{eq_03_16}.
	
	Therefore, due to \eqref{eq_03_14} and \eqref{eq_03_16}, \definitionRef{def_08} yields that
	\begin{equation}\label{eq_03_22}
		v_{\infty}\equiv0\text{ on }I.
	\end{equation}
	
	To summarise, we have obtained a sequence $\cbr{X_k}_{k=1}^\infty$ (recall the definition \eqref{eq_03_08}) of subsets of $\C{I}{E}$, each of them being non-empty, bounded, closed and convex (the first two following from the fact that $X_0$ is non-empty and bounded, the last two following directly from the definition \eqref{eq_03_08}). Moreover, $\cbr{X_k}_{k=1}^{\infty}$ satisfy \eqref{eq_03_09} and due to \eqref{eq_03_12} and \eqref{eq_03_22} also
	\begin{equation}\label{eq_03_23}
		\mu\br{X_k\br{\cdot}}\rightrightarrows0\text{ on }I.
	\end{equation}
	
	Consider the set $X_{\infty}$ defined as
	\begin{equation}\label{eq_03_24}
		X_{\infty}\defeq\bigcap_{k=1}^{\infty}X_k.
	\end{equation}
	It is obvious that $X_{\infty}$ is bounded, closed and convex (the last two following from the fact that the intersection of an arbitrary number of closed and convex sets is closed and convex).
	
	Consider the measure of non-compactness $M$ in the space $\C{I}{E}$, which is defined as in \lemmaRef{lm_04}. Let $k=1,2,\dots$ be arbitrary. Then
	\begin{equation}\label{eq_03_25}
		M\br{X_k}=\lim_{\eps\rightarrow0^+}\sup_{x\in{X}_k}\,\sup_{t,s\in{I}\,\st\abs{t-s}\leq\eps}\norm{x\br{t}-x\br{s}}_E+\sup_{t\in{I}}\br{\mu\br{X_k\br{t}}}.
	\end{equation}
	The first term on the RHS of \eqref{eq_03_25} is identically $0$, since the set $X_0\supseteq{X}_k$ is uniformly equicontinuous. Thus, from \eqref{eq_03_23} it directly follows that
	\begin{equation*}
		M\br{X_k}=\sup_{t\in{I}}\br{\mu\br{X_k\br{t}}}\xrightarrow{k\rightarrow{\infty}}0.
	\end{equation*}
	Therefore, by \axiomRef{def_04_ax_06} of \definitionRef{def_04}, $X_{\infty}\neq\emptyset$ and (in the view of \axiomRef{def_04_ax_01} of \definitionRef{def_04}) $X_{\infty}$ is relatively compact in $\C{I}{E}$.
	
	Furthermore, $\fF\br{X_{\infty}}\subseteq{X}_{\infty}$. Indeed, let $u\in{X}_{\infty}$ be arbitrary. By \eqref{eq_03_24} and the fact that for all $k$, $\fF\br{X_k}\subseteq{X}_k$ (see \eqref{eq_03_09}), we have that $\fF\br{u}\in{X}_k$, for all $k$. Thus, $\fF\br{u}\in{X}_{\infty}$. Moreover, since $X_{\infty}$ is relatively compact and $\fF\br{X_{\infty}}\subseteq{X}_{\infty}$, we obtain that $\fF\br{X_{\infty}}$ is relatively compact (in $\C{I}{E}$) as well.
	
	Now, we show that the mapping $\fF\st{X}_{\infty}\rightarrow{X}_{\infty}$ is continuous. Let $\cbr{y_k}_{k=1}^{\infty}\subseteq{X}_{\infty}$ be arbitrary such that $y_k\xrightarrow{\norm{\cdot}_{\infty}}y_0\in{X}_{\infty}$. By the Heine criterion, it is enough to show that $\fF\br{y_k}\xrightarrow{\norm{\cdot}_{\infty}}\fF\br{y_0}$. It follows that
	\begin{equation*}
		\begin{aligned}
			&\norm{\fF\br{y_k}-\fF\br{y_0}}_{\infty}=\sup_{t\in{I}}\norm{\frac{1}{\Gamma\br{\alpha}}\Bint{a}{t}{\br{t-s}^{\alpha-1}\br{f\br{s,y_k\br{s}}-f\br{s,y_0\br{s}}}\ds}}_E\\
			&\qquad\qquad\leqText{\thmRef{thm_01}}\sup_{t\in{I}}\norm{f\br{t,y_k\br{t}}-f\br{t,y_0\br{t}}}_E\cdot\sup_{t\in{I}}\br{\frac{1}{\Gamma\br{\alpha}}\Lint{a}{t}{\br{t-s}^{\alpha-1}\ds}}\\
			&\qquad\qquad=\frac{1}{\Gamma\br{\alpha+1}}\delta^{\alpha}\cdot\sup_{t\in{I}}\norm{f\br{t,y_k\br{t}}-f\br{t,y_0\br{t}}}_E\xrightarrow{k\rightarrow\infty}0,
		\end{aligned}
	\end{equation*}
	where the last limit is due to the fact that $f$ is uniformly continuous and $y_k\xrightarrow{\norm{\cdot}_{\infty}}y_0$.
	
	Finally, we can apply \theoremRef{thm_11} (in the Banach space $\C{I}{E}$), to conclude that the mapping $\fF$ has a fixed point in $X_{\infty}$. In the view of \theoremRef{thm_09}, this fixed point is a solution of \eqref{eq_03_01}. Hence proved.
	\end{proof}
	
	Next, we provide a basic existence and uniqueness theorem. Before moving on, we encourage the reader to remind themselves the notation from the very beginning of this section, namely that $\sI=\sbr{a,a+\sdelta}$, $R=\sI\times\ball{E}{u_0}{\beta}$ and $f\st{R}\rightarrow{E}$.
	\begin{theorem}
		Consider the IVP \eqref{eq_03_01}. Assume that
		\begin{enumerate}[(i)]
			\item $f$ is continuous on $R$.
			\item $\norm{f\br{t,u}}_E\leq{M}$, for all $\br{t,u}\in{R}$, for some $M>0$.
			\item $\norm{f\br{t,u}-f\br{t,v}}_E\leq\kappa\cdot\norm{u-v}_E$, for all $t\in\sI$, for all $u,v\in\ball{E}{u_0}{\beta}$, for some $\kappa>0$.\label{th_03_02_ass_03}
		\end{enumerate}
		Then there exists a unique solution of \eqref{eq_03_01} on $I\defeq\sbr{a,a+\delta}$, with $\delta$ being defined as
		\begin{equation}\label{eq_03_26}
			\delta\defeq\min\cbr{\sdelta,\br{\frac{\beta}{M}\cdot\Gamma\br{\alpha+1}}^{\frac{1}{\alpha}},\br{\frac{C}{\kappa}\cdot\Gamma\br{\alpha+1}}^{\frac{1}{\alpha}}},
		\end{equation}
		where $0<C<1$ is an arbitrary constant.
	\end{theorem}
	\begin{proof}
	Note that \eqref{eq_03_26} implies, apart from \eqref{eq_03_03}, also that
	\begin{equation}\label{eq_03_27}
		\frac{1}{\Gamma\br{\alpha+1}}\kappa\delta^{\alpha}\leq{C}.
	\end{equation}
	
	Define $X$ as
	\begin{equation*}
		X\defeq\cbr{u\st{u}\in\C{I}{E};\norm{u\br{\cdot}-u_0}_{\infty}\leq\beta}.
	\end{equation*}
	Consider the operator $\fF\st{X}\rightarrow\C{I}{E}$ defined by \eqref{eq_03_05}. By the same arguments as in \eqref{eq_03_06}, we readily obtain that $\fF\br{X}\subseteq{X}$.
	
	Now, let us verify that $\fF$ is a contraction. Let $u,v\in{X}$ be arbitrary. Then
	\begin{equation*}
		\begin{aligned}
			&\norm{\fF\br{u}-\fF\br{v}}_{\infty}\\
			&\qquad\qquad\leqText{\thmRef{thm_01}}\frac{1}{\Gamma\br{\alpha}}\cdot\sup_{t\in{I}}\br{\Lint{a}{t}{\br{t-s}^{\alpha-1}\norm{f\br{s,u\br{s}}-f\br{s,v\br{s}}}_E\,\ds}}\\
			&\qquad\qquad\leqText{\assumRef{th_03_02_ass_03}}\frac{1}{\Gamma\br{\alpha}}\kappa\cdot\sup_{t\in{I}}\br{\Lint{a}{t}{\br{t-s}^{\alpha-1}\norm{u\br{s}-v\br{s}}_E\,\ds}}\\
			&\qquad\qquad\leq\frac{1}{\Gamma\br{\alpha+1}}\kappa\cdot\norm{u-v}_{\infty}\cdot\sup_{t\in{I}}\br{t-a}^{\alpha}\leqText{\eqref{eq_03_27}}C\norm{u-v}_{\infty}.
		\end{aligned}
	\end{equation*}
	Since $0<C<1$, $\fF$ is indeed a contraction.
	
	Hence, by the Banach fixed point theorem, existence of the desired unique solution follows.
	\end{proof}

	\section{Applications}\label{sec_04}
	\setcounter{section}{4}
	\setcounter{equation}{0}
	\subsection{Kamke function of order $\alpha$}
	The first problem that arises in applications of \theoremRef{thm_14} is to find a suitable Kamke function of order $\alpha$ (recall the \definitionRef{def_08}). It is easy to verify that $w_{\alpha}\br{t,s}=H\cdot{s}$, for some constant $H>0$, satisfies the definition. This follows directly from Lemma $6.19$ in \cite{Diethelm}, which is a Grönwall-type inequality. The next theorem is an extension of this result to the superlinear case.
	
	\begin{theorem}\label{thm_16}
		Let $H>0$ and $\lambda\geq1$. Then
		\begin{equation*}
			w_{\alpha}\br{t,s}\defeq{H}\cdot{s}^{\lambda},\quad{t}\in\sbr{a,b},\,s\in\lsrbr{0,\infty},
		\end{equation*}
		is a Kamke function of order $\alpha$ on $\sbr{a,b}$.
	\end{theorem}
	\begin{proof}
	The first two axioms of \definitionRef{def_08} are obvious, thus it suffices to verify the last one. We will show that the only non-negative continuous solution of \eqref{eq_02_27} is $u\equiv0$.
	
	Let $u\st\sbr{a,b}\rightarrow\lsrbr{0,\infty}$ be a continuous function, which for all $t\in\sbr{a,b}$ satisfies
	\begin{equation}\label{eq_04_01}
		u\br{t}\leq\frac{1}{\Gamma\br{\alpha}}\cdot\Lint{a}{t}{\br{t-s}^{\alpha-1}\cdot{H}\cdot\br{u\br{s}}^{\lambda}\ds}.
	\end{equation}
	For $\eps>0$, consider the following IVP
	\begin{subequations}\label{eq_04_02}
		\begin{align}
			\DC{\alpha}{a}{y}\br{t}&=H\cdot\br{y\br{t}}^{\lambda},\quad{t}\in\sbr{a,b},\label{eq_04_02_01}\\
			y\br{a}&=\eps.
		\end{align}
	\end{subequations}
	Note that \eqref{eq_04_02} is a special case of the IVP described in \sectionRef{sec_02_04}, namely \eqref{eq_02_06} in the space $E=\R$. Apparently, all the theory described in \sectionRef{sec_02_03} and \sectionRef{sec_02_04} holds. However, we will use some specific results for the space $\R$, which can be found in \cite{Diethelm}.
	
	By Theorem $6.5$ in \cite{Diethelm}, \eqref{eq_04_02} has a unique solution on $\sbr{a,b}$ (due to the fact that $f\br{t,s}\defeq{H}\cdot{s}^{\lambda}$ defined on a compact set is Lipschitz continuous in the second variable, with the Lipschitz constant being independent of the first variable). For a given $\eps>0$, denote the unique solution of \eqref{eq_04_02} as $u_{\eps}\br{\cdot}$.
	
	By Lemma $6.2$ in \cite{Diethelm} (which is actually \theoremRef{thm_09} in the space $E=\R$), for all $t\in\sbr{a,b}$ and $\eps>0$ it holds that
	\begin{equation}\label{eq_04_03}
		u_{\eps}\br{t}=\eps+\frac{H}{\Gamma\br{\alpha}}\Lint{a}{t}{\br{t-s}^{\alpha-1}\cdot\br{u_{\eps}\br{s}}^{\lambda}\ds}.
	\end{equation}
	
	Moreover, it is easy to show that \eqref{eq_04_01} implies $u\br{a}=0$ (e.g. by similar arguments as in \eqref{eq_03_07}). Thus, for every $\eps>0$ it holds that
	\begin{equation*}
		u\br{a}=0<\eps=u_{\eps}\br{a}.
	\end{equation*}
	Since both $u\br{\cdot}$ and $u_{\eps}\br{\cdot}$ are continuous, there exists some $\eta>0$ such that for all $t\in\sbr{a,a+\eta}$ it holds that
	\begin{equation}\label{eq_04_04}
		u\br{t}<u_{\eps}\br{t}.
	\end{equation}
	We will show that \eqref{eq_04_04} holds for all $t\in\sbr{a,b}$. Assume that this is not the case, i.e. assume that there exists some $\tau\in\sbr{a,b}$ such that ${u}\br{\tau}\geq{u}_{\eps}\br{\tau}$. Define
	\begin{equation}\label{eq_04_05}
		t_0\defeq\min\cbr{t\in\sbr{a,b}\st{u}\br{t}=u_{\eps}\br{t}}.
	\end{equation}
	Note that due to continuity of $u\br{\cdot}$ and $u_{\eps}\br{\cdot}$, the minimum in \eqref{eq_04_05} exists. Thus, for all $t\in\sbr{a,t_0}$ it holds that
	\begin{equation}\label{eq_04_06}
		u\br{t}\leq{u}_{\eps}\br{t}.
	\end{equation}
	From \eqref{eq_04_06} and \eqref{eq_04_01}, for every $\eps>0$ it follows that
	\begin{equation*}
		\begin{aligned}
			u\br{t_0}&\leq\frac{H}{\Gamma\br{\alpha}}\cdot\Lint{a}{t_0}{\br{t_0-s}^{\alpha-1}\br{u_{\eps}\br{s}}^{\lambda}\ds}\\
			&<\eps+\frac{H}{\Gamma\br{\alpha}}\cdot\Lint{a}{t_0}{\br{t_0-s}^{\alpha-1}\br{u_{\eps}\br{s}}^{\lambda}\ds}\eqText{\eqref{eq_04_03}}u_{\eps}\br{t_0},
		\end{aligned}
	\end{equation*}
	which is a contradiction with \eqref{eq_04_05}. Thus, \eqref{eq_04_04} holds for all $t\in\sbr{a,b}$.
	
	Finally, note that for $\eps=0$, \eqref{eq_04_02} has the unique solution $\su\equiv0$ on $\sbr{a,b}$, due to Theorem $6.5$ in \cite{Diethelm}. Furthermore, by the proof of Theorem $6.20$ in \cite{Diethelm}, for every $\eps>0$ it holds that
	\begin{equation}\label{eq_04_07}
		\sup_{t\in\sbr{a,b}}\abs{u_{\eps}\br{t}-\su\br{t}}=\sup_{t\in\sbr{a,b}}\abs{u_{\eps}\br{t}}\leq{A}\cdot\eps,
	\end{equation}
	for some constant $A>0$. Since \eqref{eq_04_07} holds for every $\eps>0$, we get from \eqref{eq_04_04} holding on $\sbr{a,b}$ and the fact that $u\br{\cdot}$ is a non-negative continuous solution of \eqref{eq_04_01} that $u\equiv0$ on $\sbr{a,b}$.
	\end{proof}

	\begin{remark}\label{rem_16}
		Note that $w_{\alpha}\br{t,s}=h\br{t}\cdot{s}^{\lambda}$, with $h\br{\cdot}$ being continuous on $\sbr{a,b}$ and $\lambda\geq1$, is a Kamke function of order $\alpha$ as well. Indeed, using the boundedness of $h\br{\cdot}$ on $\sbr{a,b}$, \axiomRef{def_08_ax_03} of \definitionRef{def_08} is obtained from \theoremRef{thm_16}, the other axioms are obvious.
	\end{remark}
	
	\subsection{Existence of solutions in Banach sequence spaces}\label{sec_04_02}
	Now, we provide some applications of \theoremRef{thm_14} in particular Banach spaces, namely in the sequence space $c_0$. Before doing so, we encourage the reader to remind themselves the notation from the beginning of \sectionRef{sec_03}. 
	
	Firstly, we rewrite \eqref{eq_03_01} in the form of a fractional differential equation in a sequence space as
	\begin{subequations}\label{eq_04_08}
		\begin{align}
			\DC{\alpha}{a}{\bu\br{t}}&=\mathbf{f}\br{t,\bu\br{t}},\quad{t}\in\sI,\\
			\bu\br{a}&=\bvarphi,
		\end{align}
	\end{subequations}
	where $\bu\br{\cdot}\defeq\br{u_n\br{\cdot}}_{n=1}^{\infty},\mathbf{f}\br{\cdot,\bu}\defeq\br{f_n\br{\cdot,\bu}}_{n=1}^{\infty}$ and $\bvarphi=\br{\varphi_n}_{n=1}^{\infty}$.
	
	\begin{theorem}\label{thm_17}
		Consider the Banach sequence space $E=c_0$. Consider the IVP \eqref{eq_04_08}. Let  $p_j,q_j\st\sI\rightarrow\R$ be continuous and non-negative for all $j=1,2,\dots$, $\br{k_j}_{j=1}^{\infty}\subseteq\N$ be an increasing sequence such that $\lim_{j\rightarrow\infty}k_j=\infty$. Let $\lambda\geq1$ and assume that
		\begin{enumerate}[(i)]
			\item $\mathbf{f}\st{R}\rightarrow{c}_0$ is uniformly continuous on $R$.\label{th_04_02_ass_01}
			\item $\abs{f_j\br{t,\bu}}\leq{p}_j\br{t}+q_j\br{t}\cdot\sup_{i\geq{k}_j}\abs{u_i}^{\lambda}$, for all $\br{t,\bu}\in{R}$, for all $j=1,2,\dots$\label{th_04_02_ass_02}
			\item $p_j\br{t}\xrightarrow{j\rightarrow\infty}0$, for all $t\in\sI$.\label{th_04_02_ass_03}
			\item The family $\cbr{p_j\br{\cdot}}_{j=1}^{\infty}$ is uniformly bounded on $\sI$ by some $P\geq0$.
			\item The family $\cbr{q_j\br{\cdot}}_{j=1}^{\infty}$ is uniformly bounded on $\sI$ by some $Q\geq0$.\label{th_04_02_ass_05}
		\end{enumerate}
		Then there exists a solution of \eqref{eq_04_08} on $I=\sbr{a,a+\delta}$, with $\delta$ as defined in \eqref{eq_03_02} and $M$ given by
		\begin{equation}\label{eq_04_09}
			M\defeq{P}+Q\cdot\br{\norm{\bvarphi}_{c_0}+\beta}^{\lambda}.
		\end{equation}
	\end{theorem}
	\begin{proof}
	Consider the Kamke function $w_{\alpha}\br{t,s}\defeq{Q}\cdot{s}^{\lambda}$ (see \theoremRef{thm_16}) and the Hausdorff measure of non-compactness $\chi\br{\cdot}$ from \theoremRef{thm_13}. Obviously, it suffices to verify \assumptionRef{th_03_01_ass_02} and \assumptionRef{th_03_01_ass_03} from \theoremRef{thm_14}.
	
	The second assumption from \theoremRef{thm_14} is a direct consequence of \assumptionRef{th_04_02_ass_02} from \theoremRef{thm_17}, \eqref{eq_04_09} and the fact that $\bu\in\ball{c_0}{\bvarphi}{\beta}$, for all $\br{t,\bu}\in{R}$.
	
	As for the third assumption from \theoremRef{thm_14}, let $t\in\sI$ and $\emptyset\neq{X}\subseteq\ball{c_0}{\bvarphi}{\beta}$ be arbitrary. Then
	\begin{equation*}
		\begin{aligned}
			&\chi\br{f\br{t,X}}\leqText{\eqref{eq_02_21}\,\&\,\assumRef{th_04_02_ass_02}}\lim_{n\rightarrow\infty}\sup_{\bu\in{X}}\sup_{j\geq{n}}\br{p_j\br{t}+q_j\br{t}\cdot\sup_{i\geq{k}_j}\abs{u_i}^{\lambda}}\\
			&\qquad\leqText{\assumRef{th_04_02_ass_05}}\lim_{n\rightarrow\infty}\sup_{j\geq{n}}\br{p_j\br{t}}+Q\cdot\lim_{n\rightarrow\infty}\sup_{\bu\in{X}}\sup_{j\geq{n}}\sup_{i\geq{k}_j}\abs{u_i}^{\lambda}\\
			&\qquad=\limsup_{n\rightarrow\infty}\br{p_n\br{t}}+Q\cdot\br{\lim_{n\rightarrow\infty}\sup_{\bu\in{X}}\sup_{j\geq{n}}\abs{u_j}}^{\lambda}\eqText{\assumRef{th_04_02_ass_03}\,\&\,\eqref{eq_02_21}}Q\cdot\br{\chi\br{X}}^{\lambda}.
		\end{aligned}
	\end{equation*}
	Thus, applying \theoremRef{thm_14}, we obtain the desired result.
	\end{proof}
	
	\subsection{Semi-discretisation of fractional PDEs}\label{sec_04_03}
	For $u\st\sbr{a,b}\times\R\rightarrow\R$, let $\DCp{\alpha}{a}{u}{t}\br{t_0,s_0}$ denote the Caputo fractional partial derivative of order $\alpha$ of $u$ with respect to $t$ at the point $\br{t_0,s_0}$, i.e. $\DCp{\alpha}{a}{u}{t}\br{t_0,s_0}\defeq\DC{\alpha}{a}{\br{u\br{\cdot,s_0}}}\br{t_0}$.
	
	For $p\geq2$, define $\Phi_p\st\R\rightarrow\R$ as
	\begin{equation*}
		\Phi_p\br{x}\defeq\abs{x}^{p-2}x=\abs{x}^{p-1}\sgn{x},\quad{x}\in\R,
	\end{equation*}
	where $\sgn{\cdot}$ denotes the \textit{signum} function. Note that for $p=2$, $\Phi_p$ is the identity function.
	
	Let $T>0$ and define $\sI\defeq\sbr{0,T}$. Let $F\st\sI\times\lsrbr{0,\infty}\rightarrow\R$, $r\st\sI\times\lsrbr{0,\infty}\rightarrow\R$, $\psi\st\sI\rightarrow\R$ and $\varphi\st\lsrbr{0,\infty}\rightarrow\R$. Let $p\geq2$. Consider the following problem
	\begin{subequations}\label{eq_04_10}
		\begin{align}
			\DCp{\alpha}{0}{u}{t}\br{t,x}&=\diffp{}{x}\br{r\cdot\Phi_p\br{\diffp{u}{x}}}\br{t,x}+F\br{t,x},\quad{t}\in\sI,x\in\lsrbr{0,\infty},\label{eq_04_10_01}\\
			u\br{t,0}&=\psi\br{t},\quad{t}\in\sI,\\
			u\br{0,x}&=\varphi\br{x},\quad{x}\in\lsrbr{0,\infty}.
		\end{align}
	\end{subequations}
	Note that the problem \eqref{eq_04_10} consists of a fractional PDE with $p$-Laplacian, together with boundary and initial condition, respectively. Also, note that for $r\equiv1$ and $p=2$, the first term on the RHS of \eqref{eq_04_10_01} reduces to $\diffp[2]{u}{x}$.
	
	Now, we perform semi-discretisation of \eqref{eq_04_10} in the $x$-coordinate. Note that for simplicity, we consider unit step. For $n=0,1,\dots$, define $u_n\br{\cdot}\defeq{u}\br{\cdot,n}$, $\varphi_n\defeq\varphi\br{n}$, $F_n\br{\cdot}\defeq{F}\br{\cdot,n}$ and $r_{n+\frac{1}{2}}\br{\cdot}\defeq{r}\br{\cdot,n+\frac{1}{2}}$.
	
	For $n=1,2,\dots$, consider the central difference approximation of the first term on the RHS of \eqref{eq_04_10_01}
	\begin{equation}\label{eq_04_11}
		\begin{aligned}
			&\diffp{}{x}\br{r\cdot\Phi_p\br{\diffp{u}{x}}}\br{t,n}\\
			&\qquad\approx{r}_{n+\frac{1}{2}}\br{t}\cdot\Phi_p\br{u_{n+1}\br{t}-u_n\br{t}}-r_{n-\frac{1}{2}}\br{t}\cdot\Phi_p\br{u_n\br{t}-u_{n-1}\br{t}}\\
			&\qquad\eqdef\Lambda_n^p\br{t,\br{u_j\br{t}}_{j=1}^{\infty}}.
		\end{aligned}
	\end{equation}
	Using \eqref{eq_04_11}, we obtain the following semi-discrete approximation of \eqref{eq_04_10}
	\begin{subequations}\label{eq_04_12}
		\begin{align}
			\DC{\alpha}{0}{u_n}\br{t}&=\Lambda_n^p\br{t,\br{u_j\br{t}}_{j=1}^{\infty}}+F_n\br{t}\eqdef{f}_n^p\br{t,\br{u_j\br{t}}_{j=1}^{\infty}},\quad{t}\in\sI,\,n=1,2,\dots,\label{eq_04_12_01}\\
			u_0\br{t}&=\psi\br{t},\quad{t}\in\sI,\label{eq_04_12_02}\\
			u_n\br{0}&=\varphi_n,\quad{n}=1,2,\dots
		\end{align}
	\end{subequations}
	
	Considering $E$ to be an appropriate Banach sequence space, \eqref{eq_04_12} can be viewed as a special case of \eqref{eq_04_08}, namely with $a=0,\sdelta=T$ (and $\sI=\sbr{0,T}$), in the form
	\begin{subequations}\label{eq_04_13}
		\begin{align}
			\DC{\alpha}{0}{\bu}\br{t}&=\bLambda^p\br{t,\bu\br{t}}+\bF\br{t}\eqdef\mathbf{f}^p\br{t,\bu\br{t}},\quad{t}\in\sI,\\
			\bu\br{0}&=\bvarphi,
		\end{align}
	\end{subequations}
	where\,\, $\bu\br{\cdot}\defeq\br{u_n\br{\cdot}}_{n=1}^{\infty}$,\,\,\, $\bLambda^p\br{\cdot,\bu}\defeq\br{\Lambda_n^p\br{\cdot,\bu}}_{n=1}^{\infty}$,\,\,\, $\bF\br{\cdot}\defeq\br{F_n\br{\cdot}}_{n=1}^{\infty}$\,\, and $\bvarphi\defeq\br{\varphi_n}_{n=1}^{\infty}$. Namely, note that $\mathbf{f}^p\br{\cdot,\bu}=\br{f_n^p\br{\cdot,\bu}}_{n=1}^{\infty}=\br{\Lambda_n^p\br{\cdot,\bu}+F_n\br{\cdot}}_{n=1}^{\infty}$. Furthermore, note that in $\Lambda_1^p\br{\cdot,\bu\br{\cdot}}$, the term $u_0\br{\cdot}$ occurs, which is not part of $\bu\br{\cdot}=\br{u_n\br{\cdot}}_{n=1}^{\infty}$, but is rather given by \eqref{eq_04_12_02} (i.e. it can be viewed as a constant term with respect to $f_1^p\br{\cdot,\bu}$). Finally, note that $u_0$ has different meaning in \eqref{eq_03_01_02}, as it does in \eqref{eq_04_12_02}, but that should be clear from the context.
	
	Before moving on with particular example in the space $c_0$, we encourage the reader once again to remind themselves the notation used in the IVP \eqref{eq_03_01}, its reformulation \eqref{eq_04_08}, the problem \eqref{eq_04_10}, as well as in \theoremRef{thm_14}.
	
	\begin{example}\label{ex_01}
		Consider the IVP \eqref{eq_04_13} in the Banach space $E=c_0$. Assume that
		\begin{enumerate}[(i)]
			\item $\varphi\br{\cdot}$ is continuous on $\lsrbr{0,\infty}$, such that $\varphi\br{x}\xrightarrow{x\rightarrow\infty}0$.
			\item $\psi\equiv0$ on $\sI$.
			\item $r\br{\cdot,\cdot}$ is uniformly continuous and bounded on $\sI\times\lsrbr{0,\infty}$.\label{ex_01_ass_03}
			\item $F\br{\cdot,\cdot}$ is uniformly continuous and bounded on $\sI\times\lsrbr{0,\infty}$, such that for all $t\in\sI$, it holds that $F\br{t,x}\xrightarrow{x\rightarrow\infty}0$.
			\label{th_04_03_ass_04}
		\end{enumerate}
		We will show, using \theoremRef{thm_17}, that under these assumptions, \eqref{eq_04_13} has a local solution.
		
		Firstly, note that the problem \eqref{eq_04_13} is well formulated, since it can be easily shown that under given assumptions, $\mathbf{f}^p$ maps $R$ into $c_0$ and the initial value is an element of $c_0$ (i.e. $\bvarphi\in{c}_0$).
		
		Note that it is easy to verify, using the convexity of $\abs{\cdot}^{p-1}$ (recall that $p\geq2$) and Jensen's inequality, that for all $x,y\in\R$ and any $z\geq\max\cbr{\abs{x},\abs{y}}$, it holds that
		\begin{equation}\label{eq_04_14}
			\abs{\Phi_p\br{x-y}}\leq\br{\abs{x}+\abs{y}}^{p-1}\leq2^{p-2}\cdot\br{\abs{x}^{p-1}+\abs{y}^{p-1}}\leq2^{p-1}\cdot{z}^{p-1}.
		\end{equation}
		
		Now, we verify the \assumptionRef{th_04_02_ass_01} of \theoremRef{thm_17}. Let $\eps>0$ be arbitrary. Define the following constants
		\begin{align}
			C_1&\defeq\sup_{\br{t,x}\in\sI\times\lsrbr{0,\infty}}\abs{r\br{t,x}}\cdot2^p\br{p-1}\cdot\br{\norm{\bvarphi}_{c_0}+\beta}^{p-2},\label{eq_04_15}\\
			C_2&\defeq2^p\cdot\br{\norm{\bvarphi}_{c_0}+\beta}^{p-1}.\label{eq_04_16}
		\end{align}
		Let $t,s\in\sI$ and $\bu,\bv\in\ball{c_0}{\bvarphi}{\beta}$ be arbitrary. Note that due to uniform continuity of $r\br{\cdot,\cdot}$ on $\sI\times\lsrbr{0,\infty}$, there exists $\delta_2>0$ such that whenever $\abs{t-s}\leq\delta_2$, then $\abs{r\br{t,x}-r\br{s,x}}\leq\frac{\eps}{3C_2}$, for all $x\geq0$. Similarly, by the same argument, there exists $\delta_3>0$ such that whenever $\abs{t-s}\leq\delta_3$, then $\abs{F\br{t,x}-F\br{s,x}}\leq\frac{\eps}{3}$, for all $x\geq0$. Finally, let $\delta_1\defeq\frac{\eps}{3C_1}$. Then, whenever $\max\cbr{\abs{t-s},\norm{\bu-\bv}_{c_0}}\leq\min\cbr{\delta_1,\delta_2,\delta_3}$, we obtain
		\begin{equation}\label{eq_04_17}
			\begin{aligned}
				&\norm{\mathbf{f}^p\br{t,\bu}-\mathbf{f}^p\br{s,\bv}}_{c_0}\leq\sup_{n=1,2,\dots}\abs{\Lambda^p_n\br{t,\bu}-\Lambda^p_n\br{s,\bv}}+\sup_{n=1,2,\dots}\abs{F_n\br{t}-F_n\br{s}}\\
				&\qquad\leq2\cdot\sup_{n=1,2,\dots}\abs{r_{n-\frac{1}{2}}\br{t}\cdot\Phi_p\br{u_n-u_{n-1}}-r_{n-\frac{1}{2}}\br{s}\cdot\Phi_p\br{v_n-v_{n-1}}}+\\
				&\qquad\qquad\qquad+\sup_{n=1,2,\dots}\abs{F_n\br{t}-F_n\br{s}}\\
				&\qquad\leq2\cdot\sup_{n=1,2,\dots}\abs{r_{n-\frac{1}{2}}\br{t}\cdot\br{\Phi_p\br{u_n-u_{n-1}}-\Phi_p\br{v_n-v_{n-1}}}}+\\
				&\qquad\qquad\qquad+2\cdot\frac{\eps}{3C_2}\cdot\sup_{n=1,2,\dots}\abs{\Phi_p\br{v_{n}-v_{n-1}}}+\sup_{n=1,2,\dots}\abs{F_n\br{t}-F_n\br{s}}.
			\end{aligned}
		\end{equation}
		The third term on the RHS of \eqref{eq_04_17} is apparently \enquote{$\leq\frac{\eps}{3}$}. As for the second term on the RHS of \eqref{eq_04_17}, it holds that
		\begin{equation*}
			\frac{2\eps}{3C_2}\cdot\sup_{n=1,2,\dots}\abs{\Phi_p\br{v_{n}-v_{n-1}}}\leqText{\bv\in\ball{c_0}{\bvarphi}{\beta}\,\&\,\eqref{eq_04_14}}\frac{2\eps}{3C_2}\cdot2^{p-1}\cdot\br{\norm{\bvarphi}_{c_0}+\beta}^{p-1}\eqText{\eqref{eq_04_16}}\frac{\eps}{3}.
		\end{equation*}
		Finally, let $n=1,2,\dots$ be arbitrary. Using the mean value theorem, we obtain the following estimate for the first term on the RHS of \eqref{eq_04_17}
		\begin{equation*}
			\begin{aligned}
				&2\cdot\abs{r_{n-\frac{1}{2}}\br{t}\cdot\br{\Phi_p\br{u_n-u_{n-1}}-\Phi_p\br{v_n-v_{n-1}}}}\\
				&\qquad\eqText{\xi\text{ is between }\br{u_n-u_{n-1}}\text{ and }\br{v_n-v_{n-1}}}2\cdot\abs{r_{n-\frac{1}{2}}\br{t}}\cdot\br{p-1}\cdot\abs{\xi}^{p-2}\times\\
				&\qquad\qquad\qquad\times\abs{\br{u_n-u_{n-1}}-\br{v_n-v_{n-1}}}\\
				&\qquad\leqText{\bu,\bv\in\ball{c_0}{\bvarphi}{\beta}}2\cdot\abs{r_{n-\frac{1}{2}}\br{t}}\cdot\br{p-1}\cdot\br{2\br{\norm{\bvarphi}_{c_0}+\beta}}^{p-2}\cdot2\norm{\bu-\bv}_{c_0}\\
				&\qquad\leqText{\eqref{eq_04_15}}C_1\cdot\norm{\bu-\bv}_{c_0},
			\end{aligned}
		\end{equation*}
		which, considering that $\norm{\bu-\bv}_{c_0}\leq\delta_1=\frac{\eps}{3C_1}$, yields that the first term on the RHS of \eqref{eq_04_17} is \enquote{$\leq\frac{\eps}{3}$}. Hence, $\mathbf{f}^p$ is uniformly continuous on $R$.
		
		Now, defining $\lambda\defeq{p}-1$, $p_n\br{\cdot}\defeq\abs{F_n\br{\cdot}}$, $q_n\br{\cdot}\defeq2^p\cdot\max\cbr{r_{n-\frac{1}{2}}\br{\cdot},r_{n+\frac{1}{2}}\br{\cdot}}$, for $n=1,2,\dots$ and $k_1\defeq1$, $k_n\defeq{n-1}$, for $n=2,3,\dots$, one can easily see that \assumptionRef{th_04_02_ass_02} of \theoremRef{thm_17} is satisfied. Indeed, for $n=1,2,\dots$ it holds that
		\begin{equation*}
			\begin{aligned}
				\abs{f_n^p\br{t,\bu}}&\leqText{\eqref{eq_04_11}}\abs{r_{n+\frac{1}{2}}\br{t}}\cdot\abs{u_{n+1}-u_n}^{p-1}+\abs{r_{n-\frac{1}{2}}\br{t}}\cdot\abs{u_n-u_{n-1}}^{p-1}+\abs{F_n\br{t}}\\
				&\leqText{\eqref{eq_04_14}}2^p\cdot\max\cbr{r_{n-\frac{1}{2}}\br{t},r_{n+\frac{1}{2}}\br{t}}\cdot\max_{k=n-1,n,n+1}\abs{u_k}^{p-1}+\abs{F_n\br{t}}.
			\end{aligned}
		\end{equation*}
		
		The last three assumptions of \theoremRef{thm_17} are obviously satisfied (due to \assumptionRef{ex_01_ass_03} and \assumptionRef{th_04_03_ass_04} from \exampleRef{ex_01}). Thus, by \theoremRef{thm_17}, there exists a local solution of \eqref{eq_04_13}.
	\end{example}

	\section{Concluding remarks}\label{sec_05}
	\setcounter{section}{5}
	\setcounter{equation}{0}
	Most of the results that can be found in the existing literature and concerning the question of existence of solutions for problems of type \eqref{eq_03_01} require the assumption of $\mu\br{f\br{t,X}}\leq{g}\br{t}\cdot\mu\br{X}$, where $g\br{\cdot}$ is continuous and $\mu\br{\cdot}$ is either Hausdorff or Kuratowski measure of non-compactness. Our main result, \theoremRef{thm_14}, extends this assumption in two ways. Firstly, we require only $\mu\br{f\br{t,X}}\leq{w}_{\alpha}\br{t,\mu\br{X}}$, with $w_{\alpha}\br{\cdot,\cdot}$ being an arbitrary Kamke function of order $\alpha$ (see \definitionRef{def_08}). Secondly, we allow $\mu\br{\cdot}$ to be an arbitrary measure of non-compactness satisfying some additional assumptions. Note that the above-mentioned $g\br{t}\cdot\mu\br{X}$ satisfies the definition of a Kamke function of order $\alpha$ (see \theoremRef{thm_16} and \remarkRef{rem_16}). Also, note that both the Hausdorff and Kuratowski measure of non-compactness satisfy given additional assumptions (the former case being discussed in \remarkRef{rem_10}, the latter not mentioned in this paper, but one can easily find the necessary theory in \cite{Banas_book_sequences}).
	
	In \theoremRef{thm_14}, in contrast to a similar theorem for the integer order case in \cite{Banas_book} (Theorem $13.3.1$), we require the assumption of $\mu$ having the singleton property (see \definitionRef{def_06}). This is due to the fact that in \eqref{eq_03_01}, the initial condition is in a general point $u_0\in{E}$, rather than in $u_0=0$. Moreover, as mentioned in \remarkRef{rem_09}, this property does not follow from the definition of a (sublinear) measure of non-compactness in general, thus it indeed cannot be omitted.
	
	Note that even though the theory was developed for fractional differential equations of order $0<\alpha<1$, the very same arguments can be used also in the case of $\alpha=1$ (which coincides with the usual first derivative). However, note that for $\alpha=1$, \theoremRef{thm_14} is a known result (see the corresponding discussion in \sectionRef{sec_01}).
	
	There are various options for further development of our main result, \theoremRef{thm_14}. Natural question is whether some assumptions of \theoremRef{thm_14} can be weakened, at least for particular measures of non-compactness and/or particular Banach spaces. This question was discussed for the integer order case in \cite{Monch}, where it is shown that the mere continuity of $f$ (instead of the uniform continuity) is sufficient in the case of a Hausdorff measure of non-compactness for some Banach spaces. Furthermore, the question of extension of local solutions to the whole interval $\sbr{a,a+\sdelta}$ can be addressed. Also, an interesting question is a generalisation of the IVP \eqref{eq_03_01} in the context of calculus on time scales (even in the case of $\alpha=1$). Finally, the IVP \eqref{eq_03_01} can be also extended to fractional differential equations of arbitrary order $\alpha>0$.
	
	The application of \theoremRef{thm_14} in the Banach space $c_0$, namely \theoremRef{thm_17}, was inspired by \cite{Banas_paper_SolvabilityOfInifiniteSystems}, where a similar result is stated for the integer order case and $\lambda=1$. However, we believe that in \cite{Banas_paper_SolvabilityOfInifiniteSystems} (Theorem $3$), the assumption of $\cbr{p_j\br{\cdot}}_{j=1}^{\infty}$ converging uniformly to $0$ on $\sI$ can be replaced by the weaker assumption of pointwise convergence (which is only assumed in \theoremRef{thm_17}).
	
	Study of countable systems of the form \eqref{eq_04_13} (with semi-discrete $p$-Laplacian) is new in the fractional case even in the case of $p=2$ (which coincides with semi-discretisation of the usual Laplacian). Also, in the view of previous note regarding $\alpha=1$, the study of such non-fractional countable systems (i.e. for $\alpha=1$) is new in the case of $p>2$.
	
	Additionally, one can further modify the results from \sectionRef{sec_04_02} to other sequence spaces, e.g. $c$ (the space of all the convergent sequences) or $l_p$ (the space of all the absolutely summable sequences in the $p$-th power). Such systems were studied for the integer order case e.g. in \cite{Banas_paper_SolvabilityOfInifiniteSystems}.
	
	Furthermore, the fractional PDE with $p$-Laplacian in \eqref{eq_04_10_01} can be generalised to the case where $x\in\R$ (instead of $x\geq0$). Semi-discretisation of such a PDE would lead to bisequences (i.e. sequences indexed over $\Z$ instead of $\N$). Also, note that in \eqref{eq_04_10_01}, the non-linear perturbation $F\br{t,x}$ can be without major problems extended to the form $F\br{t,x,\diffp{u}{x}}$, given that the part containing $\diffp{u}{x}$ \enquote{behaves nicely}.
	
	Finally, in \exampleRef{ex_01}, the assumption of $\psi\equiv0$ can be easily extended to the case of $\psi=\psi\br{\cdot}$ being continuous on $\sI$. In such a case, the term $u_0\br{\cdot}=\psi\br{\cdot}$ from \eqref{eq_04_12_02} would appear as a (known) part of the non-linearity $f_1^p\br{\cdot,\bu}$ in \eqref{eq_04_12_01} (see also the discussion right after \eqref{eq_04_13}). One should note that in the case of continuous $\psi\br{\cdot}\not\equiv0$, the conclusion of \exampleRef{ex_01} remains the same, with only minor modifications of the arguments given in that example.

\end{document}